\documentclass[a4paper,11pt, reqno]{amsart}
\usepackage{amssymb,amsthm,amsmath}
\usepackage[T1]{fontenc}
\usepackage[colorlinks,citecolor=red,pagebackref,hypertexnames=false]{hyperref}

\numberwithin{equation}{section}
\allowdisplaybreaks[2]
\theoremstyle{plain}
\newtheorem{theorem}{Theorem}[section]

\newtheorem{lemma}[theorem]{Lemma}

\newtheorem{proposition}[theorem]{Proposition}

\theoremstyle{definition}

\theoremstyle{remark}
\newtheorem{remark}[theorem]{Remark}

\newtheorem{case[theorem]}{Case}

\def\norm#1.#2.{\lVert#1\rVert_{#2}}

\title[On local dispersive and Strichartz estimates for the Grushin operator]
{On local dispersive and Strichartz estimates for the Grushin operator
}

\author{Sunit Ghosh, Shyam Swarup Mondal and Jitendriya Swain}

\address{Sunit Ghosh, Department of Mathematics, Indian Institute of Technology Guwahati, India.}
	\email{g.sunit@iitg.ac.in}
	
	\address{Shyam Swarup Mondal,  Department of Mathematics, Indian Institute of Technology Delhi, India.}
	\email{mondalshyam055@gmail.com}

	\address{Jitendriya Swain,   Department of Mathematics, Indian Institute of Technology Guwahati, India.}
	\email{jitumath@iitg.ac.in}
	


\date{\today}

\begin{document}
\baselineskip=20pt
\markboth{} {}

\bibliographystyle{amsplain}
\title[On local dispersive and Strichartz estimates associated with ...]
{On local dispersive and Strichartz estimates associated with the Grushin operator}


\begin{abstract}
	Let  $G=-\Delta-|x|^2\partial_{t}^2$  denote  the Grushin operator  on $\mathbb{R}^{n+1}$.	The aim of this paper is two fold. In the first part, due to the non-dispersive phenomena of the   Grushin-Schr\"odinger equation on $\mathbb{R}^{n+1}$, we establish a local dispersive estimate by defining the Grushin-Schr\"odinger kernel on a suitable domain. As a corollary we obtain a local Strichartz estimate for the Grushin-Schr\"odinger equation. In the next part, we prove a restriction theorem with respect to the scaled Hermite-Fourier transform on $\mathbb{R}^{n+2}$ for certain surfaces in $\mathbb{N}_0^n\times\mathbb{R^*}\times \mathbb{R}$ and derive anisotropic Strichartz estimates for the Grushin-Schr\"{o}dinger equation and for the Grushin wave equation as well.
\end{abstract}

\maketitle
\def\BC{{\mathbb C}} \def\BQ{{\mathbb Q}}
\def\BR{{\mathbb R}} \def\BI{{\mathbb I}}
\def\BZ{{\mathbb Z}} \def\BD{{\mathbb D}}
\def\BP{{\mathbb P}} \def\BB{{\mathbb B}}
\def\BS{{\mathbb S}} \def\BH{{\mathbb H}}
\def\BE{{\mathbb E}}
\def\BN{{\mathbb N}}
\def\LP{{W(L^p(\BR^d, \BH), L^q_v)}}
\def\LPN{{W_{\BH}(L^p, L^q_v)}}
\def\LPQ{{W_{\BH}(L^{p'}, L^{q'}_{1/v})}}
\def\L1{{W_{\BH}(L^{\infty}, L^1_w)}}
\def\LB{{L^p(Q_{1/ \beta}, \BH)}}
\def\SP{S^{p,q}_{\tilde{v}}(\BH)}
\def\f{{\bf f}}
\def\h{{\bf h}}
\def\hp{{\bf h'}}
\def\m{{\bf m}}
\def\g{{\bf g}}
\def\ga{{\boldsymbol{\gamma}}}
\vspace{-.6cm}
\section{Introduction}
Consider the free Schrödinger equation on $\mathbb{R}^n$:
\begin{align}\label{eq1}
	i \partial_s u(x,s)-\Delta u(x,s) &= 0,\quad x \in \mathbb{R}^n, \hspace{2pt} s \in \mathbb{R}\setminus \{0\}, \\
	\nonumber u(x,0) &= f(x).
\end{align}
It is well known that $e^{-i s \Delta} f$ is the unique solution to the IVP (\ref{eq1}) and can be written as
\begin{align}\label{eq2}
	u(\cdot,s)=\frac{e^{i \frac{|\cdot|^2}{4 s}}}{(4 \pi i s)^{\frac{n}{2}}} * f .	
\end{align}
An application of Young's inequality in (\ref{eq2}) gives the following dispersive estimate:
\begin{align}\label{eq3}
	\forall s \ne 0, \quad \|u(\cdot,s)\|_{L^{\infty}\left(\mathbb{R}^n\right)} \leq \frac{1}{(4 \pi|s|)^{\frac{n}{2}}}\left\|f\right\|_{L^1\left(\mathbb{R}^n\right)}.
\end{align}
Such dispersive estimate is crucial in the study of semilinear and quasilinear equations which has wide applications in physical systems (see \cite{GV,Bah6} and the references therein). The dispersive estimate (\ref{eq3}) yields the following remarkable estimate for the solution of (\ref{eq1}) by Strichartz \cite{RS} in connection with Fourier restriction theory:
\begin{align}\label{eq4}
	\|u\|_{L^q\left(\mathbb{R}, L^p\left(\mathbb{R}^n\right)\right)} \leq C(p, q)\left\|f\right\|_{L^2\left(\mathbb{R}^n\right)},
\end{align}
where $(p, q)$ satisfies the following scaling admissibility condition
$
\frac{2}{q}+\frac{n}{p}=\frac{n}{2}
$
with $ p, q \geq 2 $ and $(n, q, p) \neq(2,2, \infty)$. We refer to \cite{Bour1, Bour2, Caze, KeelT} for further study on Strichartz estimates and its connection with dispersive estimates.

In this work we aim at investigating  such phenomenon associated with the Grushin operator $G$ on $\mathbb{R}^{n+1}$ defined by
 $$G=-\Delta-|x|^2\partial_{t}^2, \quad (x, t)\in \mathbb{R}^{n}\times \mathbb{R},$$
where $|x|=\sqrt{x_1^2+\cdots+x_n^2}.$

The studies of the Grushin operator date back to Baouendi and Grushin \cite{Baouendi,Grushin70,Grushin71}. Since then several authors studied the operator extensively in different contexts, involving classification of solutions to an elliptic equations, free boundary problems in partial differential equations, well-posedness problems in Sobolev spaces etc. \cite{Anh, DM, JST, Louise}. Even though numerous studies in the direction of PDEs associated with the Grushin operator are currently available, to the best of our knowledge, the study on dispersive and Strichartz estimates for the Schr\"{o}dinger operator associated with the Grushin operator has not been addressed in the literature so far.

 Consider the following free Grushin-Schr\"odinger equation:
\begin{align}\label{schr}
	i \partial_s u(x,t,s) - G u(x&,t,s) = 0,\quad s \in \mathbb{R} ,\hspace{3pt} (x,t)\in \mathbb{R}^{n+1},
	\\ \nonumber u(x,t,0) &= f(x,t).
\end{align}
For   $f$ in $L^2(\mathbb{R}^{n+1})$,  $u(x, t, s)=e^{-isG}f(x, t)$ is the unique global in time solution   to the above  IVP (\ref{schr}).

Unlike the Euclidean case, the solution to the IVP (\ref{schr})  satisfies the following.   		
\begin{proposition}\label{thm0}
	There exists a function $f\in \mathcal{S}(\mathbb{R}^{n+1})$, the set of all Schwartz class functions on $\mathbb{R}^{n+1}$, such that the solution to the IVP (\ref{schr}) with initial data $f$ satisfies
	\begin{align}
		\forall s \in \mathbb{R}, \quad \forall (x,t) \in \mathbb{R}^{n+1}, \quad u(x,t,s) = f(x,t + s n).
	\end{align}
\end{proposition}
\noindent In \cite{Gerd}, the authors  illustrate the above result for $n=1$. Notice that $\|u(\cdot,s)\|_p = \|f\|_p$ for all $1 \leq p \leq \infty$, which shows that one cannot expect for a global dispersive estimate of the type (\ref{eq2}).
Due to loss of dispersion, the Euclidean strategy of finding Strichartz estimates fails, and the problem of obtaining Strichartz estimates is considerably difficult.  For instance, we refer to similar situations for compact Riemannian manifolds \cite{burgain, Ivanovici,Burq}, for the Heisenberg group \cite{Bahouri-local, Gall}, for the hyperbolic space \cite{Anker} and for the nilpotent Lie groups \cite{Bahouri1, Del}.

In particular, Bahouri, G\'erard, and Xu \cite{Bahouri}  emphasized that the Schr\"odinger operator on the Heisenberg group $\mathbb{H}^d$  has no dispersion at all. Using the integral representation of the Schr\"odinger kernel on  horizontal Heisenberg strips, the authors obtained local dispersive estimates and as a by-product they establish local version of Strichartz estimate in \cite{Bahouri-local}. Furthermore, following the general strategy of Fourier restriction methods (see \cite{RS}) and using Fourier analysis tools on the Heisenberg group $\mathbb{H}^n$, Bahouri et. al. obtained an anisotropic  Strichartz estimates for the solutions to the linear  Schr\"odinger equation as well as the wave equation, for the radial initial data, on the Heisenberg group $\mathbb{H}^d$ involving the sublaplacian in  \cite{Gall}.

Since, Grushin operator is closely related with the sublaplacian on the Heisenberg group \cite{jyoti}, it is natural to investigate the local dispersive estimate and local Strichartz estimate for the IVP (\ref{schr}). Our main results are as follows:

Viewing $e^{-i s G}$ as an integral operator in the sense of distributions, the solution to (\ref{schr}) does not have the formulation of type (\ref{eq2}) in terms of the Grushin-Schr\"{o}dinger kernel $\mathcal{H}_s$. However, we compute the Grushin-Schr\"{o}dinger kernel locally on a strip in the following theorem:
\begin{theorem}\label{thm1}
	The kernel associated with the free  Grushin-Schr\"odinger equation (\ref{schr}) on the horizontal strip $\{(x,t;y,t_1) : |t-t_1| < n |s|\}$ for $s \ne 0$ is given by
	\begin{align}\label{bb}
		\mathcal{H}_s(x,t;y,t_1) =\frac{1}{(2 \pi s )^{\frac{n}{2}+1}}\int_{\mathbb{R}} e^{- \frac{\lambda (t-t_1)}{s}} \left(\frac{|\lambda|}{\sinh (2 |\lambda|)}\right)^{\frac{n}{2}} e^{\frac{i|\lambda|}{2 s }\left(x^{2}+y^{2}\right) \operatorname{coth}(2 |\lambda|)} e^{\frac{- i |\lambda| x\cdot y}{\hspace{.5pt} s\cdot\hspace{.5pt}\sinh (2 |\lambda | )}} d \lambda.
	\end{align}
\end{theorem}
Using the above integral representation of the Schr\"odinger kernel on horizontal strips, we establish the following local dispersive estimate.

\begin{theorem}\label{thm2}
	Given $w_0 \in \mathbb{R}^{n+1}$, let $f$ be supported in a ball $B(w_0,R_0)$ with center at $w_0$ and radius $R_0$. Then, for any positive constant $k < n$ and for all $2 \leq p \leq \infty$, the solution to the IVP (\ref{schr}) associated to the initial data $f$ satisfies the following local dispersive estimate:
	\begin{align}\label{ai}
		\|u(\cdot,s)\|_{L^p(B(w_0,\frac{1}{2} k |s| ))} \leq \left(\frac{M}{|s|^{\frac{n}{2}+1}}\right)^{1-\frac{2}{p}}\|f\|_{L^{p'}(\mathbb{R}^{n+1})},
	\end{align}
    for all $|s| \geq \frac{2 R_0}{n-k}$,~~ with  $M = \frac{1}{(2 \pi )^{\frac{n}{2}+1}} \displaystyle\int_{\mathbb{R}}\left(\frac{|\lambda|}{\sinh (2 |\lambda|)}\right)^{\frac{n}{2}} e^{\frac{n|\lambda|}{2}} d\lambda$
	and $\frac{1}{p} + \frac{1}{p'} = 1$.
\end{theorem}

Using the local dispersive estimate  (\ref{ai}),  we only obtain the local solution of (\ref{schr}) and establish the following local Strichartz estimate.
\begin{theorem}\label{thm3}
	Given $k < n$, if  $(p, q)$ lies in the admissible set  $$A_0 = \left\{(p,q) : 2 \leq p \leq\infty\quad \text{and}\quad \frac{1}{p} + \frac{1}{q} = \frac{1}{2} \right\},$$ then for all $f \in L^2(\mathbb{R}^{n+1})$ supported in the ball $B(w_0,R_0)$, with some $w_0 \in \mathbb{R}^{n+1}$, the solution to the IVP (\ref{schr}) associated to the initial data $f$ satisfies the following local Strichartz estimate:
	\begin{align}\label{bf}
		\|u\|_{L^q((-\infty,-C_k R_0])\cup [C_kR_0,\infty);L^p(B(w_0,\frac{1}{2} k s ))} \leq C(q,k) \|f\|_{L^2(\mathbb{R}^{n+1})},
	\end{align}
	where $C_k = \frac{2}{n-k}$.
\end{theorem}
The remaining part of the paper aims to establish anisotropic Strichartz estimate for the Grushin-Schr\"{o}dinger equation (\ref{schr}) and for the Grushin wave equation:
\begin{align}\label{ay}
	\partial_s^2 u&(x,t,s) + Gu(x,t,s) = 0 \quad s \in \mathbb{R} ,\hspace{3pt} (x,t)\in \mathbb{R}^{n+1}, \\
	u&(x,t,0) = f(x,t), \quad \partial_su(x,t,0) = g(x,t).\nonumber
\end{align}
In order to achieve this goal we first obtain the restriction theorem analogue to Bahouri et. al. \cite{Gall} (see also M\"{u}ller \cite{Muller} ) for  the scaled Hermite-Fourier transform (defined below) restricted to certain surfaces in $\mathbb{N}_0^n \times \mathbb{R}^* \times \mathbb{R}$. At this stage, we refer to Liu and Song \cite{Manli} for a restriction theorem associated with the Grushin operator.

We consider the mixed Lebesgue spaces $L^{r}_t(\mathbb{R};L^q_s(\mathbb{R};L^p_x(\mathbb{R}^n)))$, for $1 \leq p,q ,r \leq \infty$ with the mixed norm $$\|f\|_{L_{t}^{r} L_s^q L_{x}^{p}} = \left(\int_{\mathbb{R}}\left(\int_{\mathbb{R}}\left(\int_{\mathbb{R}^{n+1}} |f(x,t,s)|^p dx\right)^{\frac{q}{p}}ds\right)^{\frac{r}{q}}dt\right)^{\frac{1}{r}}.$$
For $f \in \mathcal{S}(\mathbb{R}^{n+2})$, the set of all Schwartz class functions on $\mathbb{R}^{n+2}$
, let
\begin{align}
	f^{\lambda,\nu}(x)=\int_{\mathbb{R}}\int_{\mathbb{R}} f(x, t, s) e^{i \lambda t}e^{i \nu s} d t ds
\end{align}
stand for the inverse Fourier transform of $ f(x, t, s)$ in the $(t, s)$ variable. We define the scaled Hermite-Fourier transform of $f$ on $\mathbb{R}^{n+2}$ as
\begin{align}\label{def2}
	\hat{f}(\alpha,\lambda,\nu)=\int_{\mathbb{R}^n} \int_{\mathbb{R}}\int_{\mathbb{R}} e^{i\lambda t} e^ {i \nu s}f(x,t,s)\Phi_\alpha^\lambda(x) ds dt  dx = \langle f^{\lambda,\nu},\Phi_\alpha^\lambda\rangle,
\end{align}
for any $(\alpha,\lambda,\nu) \in \mathbb{N}_0^n \times \mathbb{R}^* \times \mathbb{R}$. Given a surface $S$ in $ \mathbb{N}_0^n \times \mathbb{R}^* \times \mathbb{R}$ endowed with an induced measure $d \sigma$, we define the restriction operator $\mathcal{R}_S: L^2(\mathbb{R}^{n+2})\to L^2(S,d\sigma)$ defined by
\begin{align}
	\mathcal{R}_{S} f = \hat{f}|_{S},
\end{align} on the surface $S$
and the operator dual to $\mathcal{R}_{S}$ (called the extension operator) as
\begin{align}\label{ak}
	\mathcal{E}_{S} (\varTheta)(x,t,s) = \frac{1}{(2 \pi)^2}\int_{S}  e^{-i \nu s} e^{-i \lambda t} \varTheta (\alpha,\lambda,\nu)\hspace{2pt}\Phi_{\alpha}^\lambda (x)\hspace{4pt} d\sigma,
\end{align}$\varTheta\in L^2(S,d\sigma).$
Taking the surface $S = \{(\alpha,\lambda,\nu) \in \mathbb{N}_0^n \times \mathbb{R}^* \times \mathbb{R} : \nu = (2 |\alpha| + n )|\lambda| \}$ with a localized induced measure $d\sigma_{loc}$ (defined in Section \ref{sec4} precisely), we obtain the following restriction theorem for Scaled Hermite-Fourier transform.
\begin{theorem}[Scaled Hermite-Fourier restriction theorem]\label{thm4}
	If $1 \leq q \leq p < 2$, then
	\begin{align}\label{at}
		\|\mathcal{R}_{S_{loc}} f\|_{L^2(S,d\sigma_{loc})} \leq C(p,q) \|f\|_{L_t^1L^q_sL^p_x},
	\end{align}
	for all functions $f \in \mathcal{S}(\mathbb{R}^{n+2})$.
\end{theorem}
\noindent By duality, Theorem \ref{thm4} can be reframed as follows: for any $2 < p' \leq q' \leq \infty$,
	\begin{align}\label{ao}
			\|\mathcal{E}_{S_{loc}}(\varTheta)\|_{L_t^{\infty} L^{q'}_s L^{p'}_x} \leq C(p,q)\|\varTheta\|_{L^2(S,d\sigma_{loc})}
		\end{align}
holds for all $\varTheta \in L^2(S,d\sigma_{loc})$.

Now, realizing the solution of (\ref{schr}) as the extension operator $\mathcal{E}_{loc}$ acting on a suitable function on $S$, using (\ref{ao}) and the density of frequency localized functions, we prove the following anisotropic  Strichartz estimates for the solutions to the free Grushin-Schr\"odinger equation.
\begin{theorem}\label{thm5}
	Let $f\in L^2(\mathbb{R}^{n+1})$. If  $(p, q)$ lies in the admissible set
	$$A = \left\{(p,q) :  2 < p \leq q \leq \infty \quad \mbox{and} \quad \frac{2}{q} + \frac{ n}{p} = \frac{n+2}{2}  \right\},$$ then the solution $u(x,t,s) = e^{- i s G} f(x,t)$ of the IVP (\ref{schr}) is in $L^{\infty}_t(\mathbb{R};L^q_s(\mathbb{R};L^p_x(\mathbb{R}^n)))$ and satisfies the estimate:
	$$\|e^{-i s G} f(x,t)\|_{L_{t}^{\infty} L_s^q L_{x}^{p}}\leq C\|f\|_{L^2(\mathbb{R}^{n+1})}.$$
\end{theorem}
\noindent Arguing as in the Theorem \ref{thm5}, for the surface $S_0$ (defined in Remark \ref{rem1}) and making use of scaled Hermite-Fourier restriction theorem for the surface $S_0$, we prove the following anisotropic Strichartz estimate for the solution to the free Grushin wave equation.
\begin{theorem}\label{thm7}
	Let $f \in L^2(\mathbb{R}^{n+1}) $ and $G^{-1/2}g \in L^2(\mathbb{R}^{n+1})$. If  $(p, q)$ lies in the admissible set
	$$A_w = \left\{(p,q) : 2 < p \leq q \leq \infty\quad \mbox{and} \quad \frac{1}{q} + \frac{ n}{p} = \frac{n+2}{2}\right\},$$ then the solution $u(x,t,s)$ of the IVP (\ref{ay}) is in $ L^{\infty}_t(\mathbb{R};L^q_s(\mathbb{R};L^p_x(\mathbb{R}^n)))$ and satisfies the estimate:
	$$\|u(x,t,s)\|_{L_{t}^{\infty} L_s^q L_{x}^{p}}\leq C\left(\|f\|_{L^2(\mathbb{R}^{n+1})} + \|G^{-1/2}g\|_{L^2(\mathbb{R}^{n+1})} \right).$$
\end{theorem}

We prove anisotropic Strichartz estimates for the inhomogeneous Grushin-Schr\"{o}dinger equation in Theorem \ref{thm6} and for the inhomogeneous Grushin wave equation in Theorem \ref{thm8} as an application of Theorem \ref{thm5} and Theorem \ref{thm7}. Further, we discuss the validity of the restriction Theorem \ref{thm4} for $p = 2$, $1 \leq q \leq2$ in Proposition \ref{prop1} and Proposition \ref{prop2}. We also investigate the validity of the Theorems \ref{thm5}, \ref{thm7}, \ref{thm6}, \ref{thm8} for $p = 2$, $2 \leq q \leq \infty$.

We use the following notation through out this article:
\begin{itemize}
	\item $\displaystyle\sum_{\pm} f(\pm \cdot) = f(+ \cdot) + f(- \cdot)$ \quad and \quad  $\displaystyle\sum_{\pm} f(\mp \cdot) = f(- \cdot) + f(+ \cdot)$. \vspace{4pt}
	\item $\mathcal{F}_{\lambda \rightarrow s} (f(\lambda)) = \frac{1}{2 \pi} \displaystyle\int_{\mathbb{R}} e^{-i s \lambda} f(\lambda) d\lambda$.
\end{itemize}

This paper is organized as follows: We discuss about the spectral theory of the Grushin operator, properties of scaled Hermite Fourier transform on $\mathbb{R}^{n+1}$, the Grushin heat kernel and the Grushin-Schr\"odinger kernel in  Section \ref{sec2}. In section \ref{sec3} we compute the Grushin-Schr\"{o}dinger kernel on certain horizontal strips and obtain the local dispersive and local Strichartz estimate for the free Schr\"{o}dinger equation. In Section \ref{sec4} we prove the restriction theorem for the scaled Hermite Fourier transform and derive the anisotropic Strichartz estimates for the solutions to IVP (\ref{schr}) and IVP (\ref{ay}) in Section \ref{sec5}. In Section \ref{sec6} we obtain anisotropic Strichartz estimates for the inhomogeneous Grushin-Schr\"odinger equation and inhomogeneous Grushin wave equation. Finally we conclude with discussing the validity of anisotropic Strichartz estimates obtained in Section \ref{sec5} and \ref{sec6} for $p = 2$, $2 \leq q \leq \infty$ in Section \ref{sec7}.

\section{The Grushin Operator and the Grushin-Schr\"odinger kernel}\label{sec2}
In this section we discus the spectral theory for the Grushin operator and the Fourier analysis tools associated with the Grushin operator. We make use of Mehler's formula to write the integral representation of the Grushin heat kernel. We also define the Grushin-Schr\"odinger kernel in the sense of distributions.
\subsection{Spectral theory for the (scaled) Hermite and the Grushin operator:}
Let $H_k$ denote the Hermite polynomial on $\mathbb{R}$, defined by
$$H_k(x)=(-1)^k \frac{d^k}{dx^k}(e^{-x^2} )e^{x^2}, \quad k=0, 1, 2, \cdots   ,$$\vspace{.30mm}
and $h_k$ denote the normalized Hermite functions on $\mathbb{R}$ defined by
$$h_k(x)=(2^k\sqrt{\pi} k!)^{-\frac{1}{2}} H_k(x)e^{-\frac{1}{2}x^2}, \quad k=0, 1, 2, \cdots.$$\vspace{.30mm}
The higher dimensional Hermite functions denoted by $\Phi_{\alpha}$ are then obtained by taking tensor product of one dimensional Hermite functions. Thus for any multi-index $\alpha \in \mathbb{N}_0^n$ and $x \in \mathbb{R}^n$, we define
$\Phi_{\alpha}(x)=\prod_{j=1}^{n}h_{\alpha_j}(x_j).$ For $\lambda \in \mathbb{R^*} = \mathbb{R}\setminus \{0\}$, the scaled Hermite functions are defined by $\Phi^\lambda_\alpha(x) = |\lambda|^{\frac{n}{4}}$ $\Phi_{\alpha}(\sqrt{|\lambda|}x)$, they are the eigenfunctions of the (scaled) Hermite operator $H(\lambda)=-\Delta +\lambda^2 |x|^2$ with eigenvalues $(2|\alpha|+n)|\lambda|$,  where $|\alpha |=\sum_{j=1}^{n}\alpha_j$, $\alpha \in \mathbb{N}_0^n$. For each $\lambda \in \mathbb{R}^*$, the family $\{\Phi_{\alpha}^\lambda : \alpha \in \mathbb{N}_0^n\}$ is then an orthonormal basis for $L^2(\mathbb{R}^n)$. For each \(k \in \mathbb{N}\), let \(P_{k}(\lambda)\) stand for the
orthogonal projection of \(L^{2}\left(\mathbb{R}^{n}\right)\) onto the eigenspace of $H(\lambda)$ spanned by \(\left\{\Phi_{\alpha}^\lambda :|\alpha|=k\right\}\). More precisely, for $f \in L^2(\mathbb{R}^n)$
\begin{align}\label{am}
	P_k(\lambda) f = \sum_{|\alpha| = k} \langle f, \Phi_{\alpha}^\lambda\rangle \hspace{3pt}\Phi_{\alpha}^\lambda \hspace{3pt},
\end{align}
where $\langle\cdot,\cdot\rangle$ denote the standard inner product in $L^2(\mathbb{R}^n)$.
Then the spectral decomposition of \(H(\lambda)\) is explicitly given as
\begin{align}\label{ba}
	H(\lambda) f=\sum_{k=0}^{\infty}(2 k+n)|\lambda| P_{k}(\lambda) f.
\end{align}
Note that
\begin{align}\label{aaa}
	P_k(\lambda) f (x)= P_k(1)(f\circ d_{{|\lambda|}^{-\frac{1}{2}}})\circ d_{|\lambda|^{\frac{1}{2}}}(x),
\end{align}
where the dilations $d_r$ on $\mathbb{R}^{n}$ is defined by $d_r(x) = rx$ for $r > 0$.

For a Schwartz function $f	$ on $\mathbb{R}^{n+1}$, let $
f^{\lambda}(x)=\int_{\mathbb{R}} f(x, t) e^{i \lambda t} d t
$
denotes  the  inverse Fourier transform of $ f(x, t)$ in the \(t\) variable.  Then it follows that 	
\begin{align}
G f(x, t)=\frac{1}{2\pi} \int_{\mathbb{R}} e^{-i \lambda t} H(\lambda) f^{\lambda}(x) d \lambda,
\end{align}
where   the operator
$H(\lambda ) = -\Delta+ \lambda^2|x|^2 $, for $\lambda\neq 0, $  is called the (scaled) Hermite operator on $\mathbb{R}^n.$
It is well known that the Grushin operator belongs to the wide class of subelliptic operators \cite{Franchi}. Moreover, it is positive, self-adjoint, and hypoelliptic operator.
Applying the operator $G$  to the Fourier expansion $ f(x, t)=\frac{1}{2\pi}\int_{\mathbb{R}} e^{-i \lambda t} f^{\lambda}(x) d \lambda, $   we see that  $$G f(x, t)=\frac{1}{2\pi} \int_{\mathbb{R}} e^{-i \lambda t} H(\lambda) f^{\lambda}(x) d \lambda. $$
Using (\ref{ba}), the spectral decomposition of the Grushin operator is given by
\begin{align}\label{spctgr}
G f(x, t)=\frac{1}{2 \pi} \int_{\mathbb{R}} e^{-i \lambda t}\left(\sum_{k=0}^{\infty}(2 k+n)|\lambda| P_{k}(\lambda) f^{\lambda}(x)\right) d \lambda.
\end{align}
\subsection{The scaled Hermite-Fourier transform on $\mathbb{R}^{n+1}$:} For a reasonable function $f$  the scaled Fourier-Hermite transform is defined by
\begin{align}\label{def1}
	\hat{f}(\alpha,\lambda)=\int_{\mathbb{R}^n} \int_{\mathbb{R}} e^{i\lambda t}f(x,t)\Phi_\alpha^\lambda(x) dt dx = \langle f^\lambda,\Phi_\alpha^\lambda\rangle, \quad (\alpha,\lambda) \in \mathbb{N}_0^n \times \mathbb{R}^*.
\end{align}
If $f \in L^2(\mathbb{R}^{n+1})$ then $\hat{f} \in L^2(\mathbb{N}^n_0\times\mathbb{R^*})$ and satisfies the Plancherel formula
\begin{align}
	\|f\|_{L^2(\mathbb{R}^{n+1})} = \frac{1}{2 \pi} \|\hat{f}\|_{L^2(\mathbb{N}_0^n\times \mathbb{R^*})}.
\end{align}
The inversion formulae is given by
\begin{align}\label{inv1}
   f(x,t) = \frac{1}{2 \pi} \int_{\mathbb{R}} e^{-i \lambda t}\sum_{\alpha \in \mathbb{N}^n_0}\hat{f}(\alpha,\lambda)\hspace{2pt}\Phi_{\alpha}^\lambda (x )\hspace{4pt}d \lambda.
\end{align}
It  $f \in L^1(\mathbb{R}^{n+1})$, it can be seen that
\begin{align}\label{az}
	\widehat{(f\circ \delta_r)}(\alpha,\lambda) = r^{-(\frac{n}{2} + 2)} \hat{f}(\alpha,r^{-2}\lambda) ,
\end{align}
where the anisotropic dilation on $\mathbb{R}^{n+1}$ is defined by $\delta_r(x,t) = (r x,r^2t)$ for $r > 0$.

\noindent Replacing $f$ by $Gf$ in (\ref{inv1}) and comparing (\ref{def1}) with (\ref{spctgr}), we get
\begin{align}\label{shft}
	\widehat{(Gf)}(\alpha,\lambda) = (2|\alpha| + n) \hspace{1pt} |\lambda| \hspace{2pt} \hat{f}(\alpha,\lambda), \quad  (\alpha,\lambda) \in \mathbb{N}_0^n \times \mathbb{R}^*.
\end{align}
For $f \in \mathcal{S}(\mathbb{R}^{n+1})$, using (\ref{shft}) and the fact that $|\Phi_\alpha^\lambda(x) | \leq |\lambda|^{\frac{n}{4}}$, $x \in \mathbb{R}^{n+1}$, for all $N \geq 0$ we get $$|(1 + (2|\alpha| + n) \hspace{1pt} |\lambda|)^N \hat{f}(\alpha,\lambda)| \leq  |\lambda|^{\frac{n}{4}}\|(1 + G)^Nf\|_{L^1(\mathbb{R}^{n+1})},$$ for all $(\alpha,\lambda) \in \mathbb{N}_0^n \times \mathbb{R}^*$, and hence we deduce that
\begin{align}\label{zz}
	|(1 + (2|\alpha| + n) \hspace{1pt} |\lambda|)^N \hat{f}(\alpha,\lambda)| \leq |\lambda|^{\frac{n}{4}} C_{N,f}
\end{align}
for all $N \geq 0$ and for all $(\alpha,\lambda) \in \mathbb{N}_0^n \times \mathbb{R}^*$.
\subsection{The heat kernel and the Schr\"{o}dinger kernel for the  Grushin  operator}

Consider the free heat equation associated with the Grushin operator $G$:
\begin{align}\label{heat}
	\partial_s u(x,t,s) + G u(&x,t,s) = 0,\quad s \geq 0 ,\hspace{3pt} (x,t)\in \mathbb{R}^{n+1},
	\\ \nonumber u(x,t,0) &= f(x,t),
\end{align}
for an integrable function $f$ on $\mathbb{R}^{n+1}$. It is easy to check that $e^{-sG}f$ is the unique solution to the IVP (\ref{heat}). Using functional calculus for $G$, the solution $u(x,t,s)$ can be written as
\begin{align*}
  u(x,t,s) = e^{-sG} f &(x,t)
  = \frac{1}{2\pi} \int_{\mathbb{R}} e^{-i \lambda t}\sum_{\alpha \in \mathbb{N}^{n}} e^{-s(2|\alpha|+n)|\lambda|} \hat{f}(\alpha,\lambda)\Phi_{\alpha}^\lambda(x) d\lambda
\\  &= \int_{\mathbb{R}^{n}}\int_{\mathbb{R}}\left(\frac{1}{2 \pi} \int_{\mathbb{R}} e^{-i\lambda(t-t_1)} \sum_{\alpha \in \mathbb{N}^{n}} e^{-s(2|\alpha|+n)|\lambda|} \Phi_{\alpha}^\lambda(x)\Phi_{\alpha}^\lambda(y) d\lambda \right) f(y,t_1) dy dt_1
 \\ &= \int_{\mathbb{R}^{n}}\int_{\mathbb{R}} K_s(x,t;y,t_1) f(y,t_1) dy dt_1,
\end{align*}
where $ K_s(x,t;y,t_1)$ is the Grushin heat kernel given by
\begin{align} \label{ker1}
	 K_s(x,t;y,t_1) = \frac{1}{2 \pi}\int_{\mathbb{R}} e^{-i\lambda(t-t_1)} \sum_{\alpha \in \mathbb{N}^{n}} e^{-s(2|\alpha|+n)|\lambda|} \Phi_{\alpha}^\lambda(x)\Phi_{\alpha}^\lambda(y) d\lambda.
\end{align}
In view of Mehler's formula \cite{Than}, the series in (\ref{ker1}) can be further simplified and we write
 \begin{align}\label{0001}
	{K}_s ({x,t;  y, t_1})=\frac{1}{(2\pi)^{\frac{n}{2}+1}}\int_{\mathbb{R}} e^{-i \lambda (t-t_1)} \left(\frac{|\lambda|}{\sinh (2 s|\lambda|)}\right)^{\frac{n}{2}} e^{\frac{-|\lambda|}{2}\left(x^{2}+y^{2}\right) \operatorname{coth}(2s|\lambda|)} e^{\frac{|\lambda| x\cdot y}{\sinh (2s |\lambda | )}} d \lambda.
\end{align}
Performing the change of variable $s\lambda \mapsto \lambda$, the heat kernel can also be written as,
\begin{align}\label{0001a}
	{K}_s ({x,t;  y, t_1})=\frac{1}{(2 \pi s )^{\frac{n}{2}+1}}\int_{\mathbb{R}} e^{-i \frac{\lambda (t-t_1)}{s}} \left(\frac{|\lambda|}{\sinh (2 |\lambda|)}\right)^{\frac{n}{2}} e^{\frac{-|\lambda|}{2 s }\left(x^{2}+y^{2}\right) \operatorname{coth}(2 |\lambda|)} e^{\frac{|\lambda| x\cdot y}{\hspace{.5pt} s\cdot\hspace{.5pt}\sinh (2 |\lambda | )}} d \lambda.
\end{align}

As in the case of the heat equation for the Grushin operator, using functional calculus for $G$ the unique solution of the IVP (\ref{schr}) is given by
\begin{align}
	u(x,t,s) = e^{-i s G} f &(x,t)
	= \frac{1}{2\pi} \int_{\mathbb{R^*}} e^{-i \lambda t}\sum_{\alpha \in \mathbb{N}^{n}} e^{-i s (2|\alpha|+n)|\lambda|} \hat{f}(\alpha,\lambda)\Phi_{\alpha}^\lambda(x) d\lambda. \label{ap}
\end{align}
and defined in the sense of tempered distribution  as,
\begin{align}
 \int_{\mathbb{R}^{n+1}} u(x,t,s)\varphi(x,t) dx dt	&= \int_{\mathbb{R}^{n}}\int_{\mathbb{R}}\left(\frac{1}{2 \pi}  \int_{\mathbb{R^*}} e^{i\lambda t_1}  \sum_{\alpha \in \mathbb{N}^{n}} e^{-i s (2|\alpha|+n)|\lambda|} \hat{\varphi}(\alpha,-\lambda) \Phi_{\alpha}^\lambda(y) d\lambda \right) f(y,t_1) dy dt_1 \nonumber
\\ &= \int_{\mathbb{R}^{n}}\int_{\mathbb{R}} \left(\int_{\mathbb{R}^{n+1}}\mathcal{H}_s(x,t;y,t_1)\varphi(x,t) dx dt \right) f(y,t_1) dy dt_1, \nonumber
\end{align}
for $\varphi \in \mathcal{S}(\mathbb{R}^{n+1})$, where $ \mathcal{H}_s(x,t;y,t_1)$ is the Grushin-Schr\"odinger kernel given by
\begin{align} \label{ker2}
	\int_{\mathbb{R}^{n+1}}\mathcal{H}_s(x,t;y,t_1)\varphi(x,t) dx dt =\frac{1}{2 \pi}  \int_{\mathbb{R^*}} e^{i\lambda t_1}  \sum_{\alpha \in \mathbb{N}^{n}} e^{-i s (2|\alpha|+n)|\lambda|} \hat{\varphi}(\alpha,-\lambda) \Phi_{\alpha}^\lambda(y) d\lambda.
\end{align}
We proceed to prove Proposition \ref{thm0}.

\noindent{\bf \emph{Proof of Proposition \ref{thm0}} : }
Fix a function $Q \in C_c^\infty((1,\infty))$ and consider
\begin{align}\label{aj}
	f(x,t) = \frac{1}{2 \pi} \int_{1}^{\infty} e^{-i \lambda t} \Phi_0^\lambda(x) Q(\lambda) d\lambda.
\end{align} Thus $f\in\mathcal{S}(\mathbb{R}^{n+1})$ and
comparing (\ref{aj}) with the inversion formula (\ref{inv1}) we have
$$
\hat{f}(\alpha,\lambda)=
\left\{\begin{array}{lr}
        0, & \text{if } \alpha\neq 0, \alpha\in\mathbb{R}^*\\
        Q(\lambda), & \hspace{2pt}\text{if } \alpha = 0, \alpha\in\mathbb{R}^*.
                \end{array}\right.
$$
By (\ref{inv1}), the solution of the IVP (\ref{schr}) can be written as
\begin{align*}
	u(x,t,s) = e^{-isG}f(x,t) = \frac{1}{2 \pi} \int_{1}^{\infty} e^{-i \lambda (t + ns)} \Phi_0^\lambda(x) Q(\lambda) d\lambda = f(x,t + ns).
\end{align*}
$\hfill\square$

\section{Local dispersive and local Strichartz estimate}\label{sec3}
In order to establish the local dispersive and local Strichartz estimate we need to compute the Grushin-Schr\"{o}dinger kernel on strips.

Observe that $ z \mapsto H_1(z) = \frac{1}{(2 \pi)^n} \int_{\mathbb{R}^n} e^{i x \xi} e^{- z |\xi|^2} d\xi \quad~\mbox{and}~\hspace{3pt} z \mapsto H_2(z) = \frac{1}{(4 \pi z)^{\frac{n}{2}}} e^{- \frac{|x|^2}{4 z}}$ are holomorphic functions on the Right half plane. Since $H_1(z) = H_2(z)$ when $z \in \mathbb{R}$ and $Re (z) > 0$, $H_1(z) = H_2(z)$ in the Right half plane. Let $\alpha$ be a non-zero purely imaginary number and let $w_n$ be a sequence in the Right half plane converging to $\alpha$. Then an application of dominated convergence theorem ensures that $H_1(w_n)$ and $H_2(w_n)$ converges in the sense of distribution. Based on this observation the Schr\"odinger kernel can be defined for the Euclidean case. We use a similar idea to compute the Grushin-Schr\"odinger kernel on the horizontal strip defined in Theorem \ref{thm1}.
\begin{proposition}
	Fix  $0 < C < n$, the maps
	$$z \mapsto K^1_z(x,t;y,t_1) \quad and \quad z \mapsto K^2_z(x,t;y,t_1)$$
	\begin{align}
		K^1_z(x,t;y,t_1) &= \frac{1}{2 \pi} \int_{\mathbb{R}^*} e^{-i\lambda(t-t_1)} \sum_{\alpha \in \mathbb{N}^{n}} e^{-z (2|\alpha|+n)|\lambda|} \Phi_{\alpha}^\lambda(x)\Phi_{\alpha}^\lambda(y) d\lambda, \label{aa}  \\
		K^2_z(x,t;y,t_1) &= \frac{1}{(2 \pi z )^{\frac{n}{2}+1}}\int_{\mathbb{R}^*} e^{-i \frac{\lambda (t-t_1)}{z}} \left(\frac{|\lambda|}{\sinh (2 |\lambda|)}\right)^{\frac{n}{2}} e^{\frac{-|\lambda|}{2 z }\left(x^{2}+y^{2}\right) \operatorname{coth}(2 |\lambda|)} e^{\frac{|\lambda| x\cdot y}{\hspace{.5pt} z\cdot\hspace{.5pt}\sinh (2 |\lambda | )}} d \lambda \label{ab},
	\end{align}
	for all $(x,t), (y,t_1) \in \mathbb{R}^{n+1}$, are holomorphic on $D = \{z \in \mathbb{C} : Re(z) > 0\}$ and $\tilde{D}_{|t-t_1|} = \{z \in \mathbb{C} : |z| > \frac{|t-t_1|}{C} \hspace{5pt} \mbox{and}~ \hspace{5pt} Re(z) > 0 \}$ respectively. Moreover, $K^1_z = K^2_z$ on the whole domain $\tilde{D}_{|t-t_1|}$.
\end{proposition}
\begin{proof}
	Interchanging summation and integration in (\ref{aa}) and performing the change of variable $(2|\alpha| + n)\lambda \mapsto \lambda $ in each integral, we get
	\begin{align*}
		K^1_z(x,t;y,t_1) = \frac{1}{2 \pi}&\sum_{\alpha \in \mathbb{N}^{n}} \frac{1}{(2|\alpha| + n)^{\frac{n}{2}+1}} \\
		&\times \int_{\mathbb{R}} e^{-i\frac{\lambda(t-t_1)}{2 |\alpha| + n}}  e^{-z |\lambda|} \Phi_{\alpha}\left(\sqrt{\frac{|\lambda|}{2 |\alpha| + n}} x \right) \Phi_{\alpha}\left(\sqrt{\frac{|\lambda|}{2 |\alpha| + n} } y\right) |\lambda|^{\frac{n}{2}}d\lambda,
	\end{align*}
	where each term of the above summation is an entire function. Since, $|\Phi_\alpha(x)| \leq 1$ uniformly, we obtain
	\begin{align*}
		\int_{\mathbb{R}} \left|e^{-i\frac{\lambda(t-t_1)}{2 |\alpha| + n}}  e^{-z |\lambda|} \Phi_{\alpha}\left(\sqrt{\frac{|\lambda|}{2 |\alpha| + n}} x \right) \Phi_{\alpha}\left(\sqrt{\frac{|\lambda|}{2 |\alpha| + n} } y\right) \right| |\lambda|^{\frac{n}{2}}d\lambda \leq C \int_{\mathbb{R}} e^{-a|\lambda|} |\lambda|^{\frac{n}{2}} d\lambda < \infty
	\end{align*}
	and
	\begin{align*}
		\int_{\mathbb{R}} \left|\partial_z\left( e^{-i\frac{\lambda(t-t_1)}{2 |\alpha| + n}}  e^{-z |\lambda|} \Phi_{\alpha}\left(\sqrt{\frac{|\lambda|}{2 |\alpha| + n}} x \right) \Phi_{\alpha}\left(\sqrt{\frac{|\lambda|}{2 |\alpha| + n} } y\right) \right) \right| & |\lambda|^{{\frac{n}{2}}+1}d\lambda \\
		&\leq C \int_{\mathbb{R}} e^{-a|\lambda|} |\lambda|^{\frac{n}{2}+1} d\lambda < \infty,
	\end{align*}
	for all $z \in \mathbb{C}$ with $Re(z) \geq a > 0$, By Lebesgue derivation theorem the map $z \mapsto K^1_z$ is holomorphic on the domain $D$.
	
	Again (\ref{ab}) can be re-written as
	\begin{align}\label{ac}
		K^2_z(x,t;y,t_1) = \frac{1}{(2 \pi z )^{\frac{n}{2}+1}}\int_{\mathbb{R}} \left(\frac{|\lambda|}{\sinh (2 |\lambda|)}\right)^{\frac{n}{2}}  e^{-i \frac{\lambda (t-t_1)}{z}} e^{-\frac{|\lambda|X(x,y,\lambda)}{2 z }}  d \lambda,
	\end{align}
	where $X(x,y,\lambda) = \frac{\left(x^{2}+y^{2}\right) \operatorname{cosh}(2 |\lambda|) - 2 x\cdot y}{\sinh (2|\lambda|)} \geq 0 $ for all $\lambda \in \mathbb{R}$ and $x,y \in \mathbb{R}^n$.
	Setting $z = |z|e^{i \arg(z) }$, we have
	\begin{align*}
		| e^{-\frac{|\lambda|X(x,y,\lambda)}{2 z }}| = e^{-\frac{|\lambda|X(x,y,\lambda)}{2 |z| }\cos \hspace{1pt}(\arg (z))} \quad \mbox{and} \quad
		|\partial_z e^{-\frac{|\lambda|X(x,y,\lambda)}{2 z }}| =\frac{|\lambda| X(x,y,\lambda)}{2 |z|^2} e^{-\frac{|\lambda|X(x,y,\lambda)}{2 |z| }\cos \hspace{1pt}(\arg (z))}.
	\end{align*}
	Also one can easily check that
	\begin{align*}
		|e^{-i \frac{\lambda (t-t_1)}{z}}| = e^{- \frac{\lambda (t-t_1) \sin \hspace{1pt} (\arg(z))}{|z|}} \quad \mbox{and} \quad
		|\partial_z e^{-i \frac{\lambda (t-t_1)}{z}}| = \frac{|t-t_1|\hspace{1pt}|\lambda|}{|z|^2}e^{- \frac{\lambda (t-t_1) \sin \hspace{1pt} (\arg(z))}{|z|}}.
	\end{align*}
	Hence for all $\lambda \in \mathbb{R}^* $, $(x,t),(y,t_1) \in \mathbb{R}^{n+1}$ and  all $z \in \mathbb{C}$ with $Re (z) \geq a >0$,
	\begin{align*}
		\left|e^{-i \frac{\lambda (t-t_1)}{z}-\frac{|\lambda|X(x,y,\lambda)}{2 z }}\right| \leq e^{\frac{|\lambda|\hspace{1pt} |t-t_1|}{|z|}} ~~~\mbox{and}~~~
		\left|\partial_z \left( e^{-i \frac{\lambda (t-t_1)}{z}-\frac{|\lambda|X(x,y,\lambda)}{2 z }} \right)\right| \leq e^{ \frac{|\lambda|\hspace{1pt} |t-t_1|}{|z|}} \left( \frac{|\lambda|\hspace{1pt}|t-t_1|}{|z|^2} + \frac{1}{a}\right)
	\end{align*}
	Taking $0 < C < n$ and combining formula (\ref{ac}) together with the Lebesgue derivation theorem, we deduce that the map $z \mapsto K^2_z$ is holomorphic on
	$\tilde{D}_{|t-t_1|}$.
	By (\ref{ker1}) and (\ref{0001a}), the maps $K^1_z$ and $K^2_z$ coincide on the intersection of the real line with $\tilde{D}_{|t-t_1|}$ , we conclude that $K^1_z = K^2_z$ on  $\tilde{D}_{|t-t_1|}$.
\end{proof}
With this information, we prove Theorem \ref{thm1} using the notations used in Proposition \ref{thm0}.
\noindent{\bf \emph{Proof of Theorem \ref{thm1}}:}
Choose a sequence $(z_p)_{p \in \mathbb{N}}$ of elements in $\tilde{D}_{|t-t_1|}$ which converges to $ i s $, for some $s \in \mathbb{R}^*$. For $f \in \mathcal{S}(\mathbb{R}^{n+1})$,
\begin{align}\label{ad}
	\int_{\mathbb{R}^{n}}\int_{\mathbb{R}} K^1_{z_p}(x,t;y,t_1) f(y,t_1) dy dt_1 = \frac{1}{2\pi} \sum_{\alpha \in \mathbb{N}_0^{n}}\int_{\mathbb{R}} e^{-i \lambda t} e^{-z_p (2|\alpha|+n)|\lambda|} \hat{f}(\alpha,\lambda)\Phi_{\alpha}^\lambda(x) d\lambda.
\end{align}
 From (\ref{zz}), we have
 $|\hat{f}(\alpha,\lambda)| \leq C_{N,f} |\lambda|^{\frac{n}{4}} (1 + (2|\alpha| + n)|\lambda|)^{-N}$, for $N \geq n + 2$.
 Therefore,
\begin{align*}
\sum_{\alpha \in \mathbb{N}_0^{n}}\int_{\mathbb{R}} \left|e^{-i \lambda t} e^{-z_p (2|\alpha|+n)|\lambda|} \hat{f}(\alpha,\lambda) \Phi_{\alpha}^\lambda(x)\right| &d\lambda \\
\leq C_{N,f} &\sum_{\alpha \in \mathbb{N}_0^n} \frac{1}{(2|\alpha|+ n)^{\frac{n}{2} + 1}} \int_{\mathbb{R}} (1 + |\lambda|)^{-N} |\lambda|^{{\frac{n}{2}}} d\lambda < \infty,
\end{align*}
where the last term obtained by performing the change of variables $(2|\alpha| + n)\lambda \mapsto \lambda$ in each integral. Applying Lebesgue dominated convergence theorem, we get
\begin{align}\label{ae}
    \lim_{p \rightarrow\infty} 	\int_{\mathbb{R}^{n}}\int_{\mathbb{R}} K^1_{z_p}(x,t;y,t_1) f(y,t_1) dy dt_1 
     =  e^{-isG}f(x,t).
\end{align}
Using the fact that $|z_p| > \frac{|t-t_1|}{C}$ and $0 < C < n$ together with (\ref{ac}), gives
$$\int_{\mathbb{R}} \left| \left(\frac{|\lambda|}{\sinh (2 |\lambda|)}\right)^{\frac{n}{2}}  e^{-i \frac{\lambda (t-t_1)}{z_p}} e^{-\frac{|\lambda|X(x,y,\lambda)}{2 z_p }} \right| d \lambda  \leq \int_{\mathbb{R}} \left(\frac{|\lambda|}{\sinh (2 |\lambda|)}\right)^{\frac{n}{2}}  e^{C|\lambda|}  d \lambda < \infty.$$
Applying Lebesgue dominated convergence theorem, we deduce that \begin{align}\label{af}
	\lim_{p\rightarrow\infty} K^2_{z_p}(x,t;y,t_1) = \frac{1}{(2 \pi s )^{\frac{n}{2}+1}}\int_{\mathbb{R}} e^{- \frac{\lambda (t-t_1)}{s}} \left(\frac{|\lambda|}{\sinh (2 |\lambda|)}\right)^{\frac{n}{2}} e^{\frac{i|\lambda|}{2 s }\left(x^{2}+y^{2}\right) \operatorname{coth}(2 |\lambda|)} e^{\frac{- i |\lambda| x\cdot y}{\hspace{.5pt} s\cdot\hspace{.5pt}\sinh (2 |\lambda | )}} d \lambda
\end{align}
for all $(x,t),(y,t_1) \in \mathbb{R}^{n+1}$ satisfying $|t-t_1| < n |s|$.
The proof of Theorem \ref{thm1} follows from (\ref{ae}), (\ref{af}) and the fact that $K^1_{z_p} = K^2_{z_p}$ on $\tilde{D}_{|t-t_1|}$.
$\hfill\square$


\noindent{\bf \emph{Proof of Theorem \ref{thm2}}:} Since the linear Schr\"{o}dinger equation on $\mathbb{R}^{n+1}$ is invariant by translation, it suffices to prove the result for $w_0 = 0$. Let $f$ is supported in $B(w_0,R_0)$, the solution to the IVP (\ref{schr}) with the given initial data $f$ is given by
\begin{align}\label{ag}
	u(x,t,s)  =	\int_{\mathbb{R}^{n}}\int_{\mathbb{R}} \mathcal{H}_s(x,t;y,t_1) f(y,t_1) dy dt_1
\end{align}
on any ball $B(0,\frac{1}{2} k |s| )$ with $0 < k < n$, where
\begin{align*}
\mathcal{H}_s(x,t;y,t_1) =\frac{1}{(2 \pi s )^{\frac{n}{2}+1}}\int_{\mathbb{R}} e^{- \frac{\lambda (t-t_1)}{s}} \left(\frac{|\lambda|}{\sinh (2 |\lambda|)}\right)^{\frac{n}{2}} e^{\frac{i|\lambda|}{2 s }\left(x^{2}+y^{2}\right) \operatorname{coth}(2 |\lambda|)} e^{\frac{- i |\lambda| x\cdot y}{\hspace{.5pt} s\cdot\hspace{.5pt}\sinh (2 |\lambda | )}} d \lambda:
\end{align*}
Since for any $(x,t) \in B(0,\frac{1}{2} k |s| )$ and any $(y,t_1) \in B(0,R_0)$, we have
$$|t-t_1| < \frac{1}{2}\hspace{1pt} k \hspace{1pt} |s| + R_0 < \frac{1}{2}\hspace{1pt} n\hspace{1pt} |s| < n \hspace{1pt}  |s|$$
provided that $|s| >  \frac{2 R_0}{n-k}.$
Note that
\begin{align}
  \|\mathcal{H}_s\|_{L^\infty(B(0,\frac{1}{2} k |s| )\times B(0,R_0))} \leq \frac{1}{(2 \pi |s| )^{\frac{n}{2}+1}} \int_{\mathbb{R}}\left(\frac{|\lambda|}{\sinh (2 |\lambda|)}\right)^{\frac{n}{2}} e^{\frac{n|\lambda|}{2}} d\lambda :=  \frac{M}{|s| ^{\frac{n}{2}+1}},
\end{align}
for $|s| >  \frac{2 R_0}{n-k}$. Using (\ref{ag}), we obtain the following $L^1 - L^\infty$ estimate,
\begin{align}\label{bc}
	\|u(\cdot,s)\|_{L^\infty (B(w_0,\frac{1}{2} k |s| ))} \leq \frac{M}{|s|^{\frac{n}{2}+1}} \|f\|_{L^{1}(\mathbb{R}^{n+1})}.
\end{align}

Furthermore, using the unitarity of $e^{-isG}$, we have
\begin{align}\label{bd}
	\|u(\cdot,s)\|_{L^2 (B(w_0,\frac{1}{2} k |s| ))} \leq	\|u(\cdot,s)\|_{L^2(\mathbb{R}^{n+1})} = \|f\|_{L^{2}(\mathbb{R}^{n+1})}.
\end{align}
Interpolating (\ref{bc}) and (\ref{bd}), for all $2 \leq p \leq \infty$ and for all $|s| >  \frac{2 R_0}{n-k}$,  we get (\ref{ai}).

$\hfill\square$

As a by-product of (\ref{ai}), we obtain the local Strichartz estimate for the IVP (\ref{schr}) under certain admissible condition for the pair $(p,q)$.

\noindent{\bf \emph{Proof of Theorem \ref{thm3}} : }
Since $f$ is supported in $B(w_0,R_0)$, applying the H\"{o}lder inequality, we obtain $\|f\|_{L^{p'}(\mathbb{R}^{n+1})} \leq R_0^{(n+1)(\frac{1}{2}-\frac{1}{p})} \|f\|_L^2(\mathbb{R}^{n+1})$ and hence for all $2 \leq p \leq \infty$, (\ref{ai}) becomes
\begin{align*}
	\|u(\cdot,s)\|_{L^p(B(w_0,\frac{1}{2} k s ))} \leq C(k) \frac{R_0^{(n+1)(\frac{1}{2}-\frac{1}{p})}}{|s|^{(n+2)(\frac{1}{2}-\frac{1}{p})}}\|f\|_{L^{2}(\mathbb{R}^{n+1})},
\end{align*}
for all $|s| \geq C_k R_0$. Using the fact $\frac{1}{p} + \frac{1}{q} = \frac{1}{2}$ we get (\ref{bf}).
$\hfill\square$

\section{Restriction theorem on $\mathbb{N}^n_0\times\mathbb{R}^*\times\mathbb{R}$}\label{sec4}
For a Schwartz class function $f$ on $\mathbb{R}^{n+2}$, the scaled Hermite-Fourier transform of $f$ on $\mathbb{R}^{n+2}$ is defined by
\begin{align}\label{def2}
	\hat{f}(\alpha,\lambda,\nu)=\int_{\mathbb{R}^n} \int_{\mathbb{R}}\int_{\mathbb{R}} e^{i\lambda t} e^ {i \nu s}f(x,t,s)\Phi_\alpha^\lambda(x) ds dt  dx = \langle f^{\lambda,\nu},\Phi_\alpha^\lambda\rangle,
\end{align}
for any $(\alpha,\lambda,\nu) \in \mathbb{N}_0^n \times \mathbb{R}^* \times \mathbb{R}$.
If $f \in L^2(\mathbb{R}^{n+2})$ then $\hat{f} \in L^2(\mathbb{N}^n_0\times\mathbb{R^*}\times \mathbb{R})$ and satisfies the Plancherel formula
\begin{align}\label{zv}
	\|f\|_{L^2(\mathbb{R}^{n+2})} = \frac{1}{(2 \pi)^2} \|\hat{f}\|_{L^2(\mathbb{N}_0^n\times \mathbb{R^*}\times \mathbb{R})}.
\end{align}
The inversion formula is given by
\begin{align}\label{inv2}
	f(x,t,s) = \frac{1}{(2 \pi)^2}\int_{\mathbb{R}} \int_{\mathbb{R}} e^{-i \nu s} e^{-i \lambda t} \sum_{\alpha \in \mathbb{N}^n_0}\hat{f}(\alpha,\lambda,\nu)\hspace{2pt}\Phi_{\alpha}^\lambda (x )\hspace{4pt}d \lambda  d\nu.
\end{align}

\subsection{A surface measure:}\label{aq} Let us consider the surface
\begin{align}\label{sur}
	S = \{(\alpha,\lambda,\nu) \in \mathbb{N}_0^n \times \mathbb{R}^* \times \mathbb{R} : \nu = (2 |\alpha| + n )|\lambda| \}.
\end{align}
We endow $S$ with the measure $d\sigma$ induced by the projection $\pi : \mathbb{N}_0^n \times \mathbb{R}^* \times \mathbb{R} \rightarrow\mathbb{N}_0^n \times \mathbb{R}^* $ onto the first two factors, where $\mathbb{N}_0^n \times \mathbb{R}^*$ endowed with the measure $d\mu\otimes d\lambda$, $d\mu$ and $d\lambda$ denote the counting measure on $\mathbb{N}^n_0$ and Lebesgue measure on $\mathbb{R^*}$ respectively. More explicitly, for any integrable function $\varTheta$ on $S$ we have
$$\int_{S} \varTheta \hspace{3pt} d\sigma = \sum_{\alpha \in \mathbb{N}^n_0}\int_{\mathbb{R^*}} \varTheta (\alpha,\lambda,(2|\alpha|+n)|\lambda|) \hspace{2pt} d\lambda.$$
By construction it is clear that if $\Theta = \hat{f}\circ \pi|_S$, where $\hat{f}$ is a function on $\mathbb{N}_0^n \times \mathbb{R}^* \times \mathbb{R}$, then for all $1\leq p \leq \infty$
\begin{align}\label{zh}
	\|\Theta\|_{L^p(S,d\sigma)} = \|\hat{f}\|_{L^p(\mathbb{N}_0^n \times \mathbb{R}^* \times \mathbb{R})}.
\end{align}
In view of Fourier restriction theorem for smooth compact surfaces due to Tomas \cite{PT} in the Euclidean space, (\ref{zh}) is well-defined for an appropriate function $f$ and holds good for the compact subsets of $S$. Therefore, we consider the surface $S$ endowed with the surface measure $d\sigma_{loc} = \psi(\nu) d\sigma$ defined by
\begin{align}\label{al}
	\int_{S} \varTheta \hspace{3pt} d\sigma_{loc} = \sum_{\alpha \in \mathbb{N}^n_0}\int_{\mathbb{R^*}} \varTheta (\alpha,\lambda,(2|\alpha|+n)|\lambda|)\hspace{2pt} \psi((2|\alpha| + n)|\lambda|) \hspace{2pt}d\lambda.
\end{align}
with $\psi$ any smooth, even, compactly supported function in $\mathbb{R}$ with an $L^\infty$ norm at most 1.

The restriction operator, $\mathcal{R}_{S_{loc}}$ and the extension operator, $\mathcal{E}_{S_{loc}}$ with respect to the surface $(S,d\sigma_{loc})$ can be computed as $\mathcal{R}_{S_{loc}} f = \hat{f}|_{S}$ and
\begin{align}\label{ak}
	\mathcal{E}_{S_{loc}} (\varTheta)(x,t,s) = \frac{1}{(2 \pi)^2}\sum_{\alpha \in \mathbb{N}^n_0}\int_{\mathbb{R^*}}  e^{-i (2|\alpha|+n)|\lambda| s} e^{-i \lambda t} \varTheta (\alpha,\lambda,\nu)\hspace{2pt}\Phi_{\alpha}^\lambda (x )\hspace{4pt} d\lambda.
\end{align}
\subsection{Restriction theorem:}
We are in a position to prove Theorem \ref{thm4}. Before proceeding to prove we need to observe the following:
\begin{lemma}\label{lem1}
	Let $ \phi \in \mathcal{S}(\mathbb{R}^n)$ and $\lambda \in \mathbb{R^*}$, then for all $1 \leq p \leq 2$,
	\begin{align}
		\|P_k(\lambda)\phi\|_{L^{p'}(\mathbb{R}^n)} \leq C |\lambda|^{\frac{n}{2}(1-\frac{2}{p'})}(2k + n)^{\frac{n-1}{2}(1-\frac{2}{p'})} \|\phi\|_{L^p(\mathbb{R}^n)},
	\end{align}
where $p'$ is the conjugate exponent of $p$, i.e., $\frac{1}{p} + \frac{1}{p'} = 1$.
\end{lemma}

\begin{proof}
	Since, $P_k(\lambda)$'s are orthogonal projections on $L^2(\mathbb{R}^n)$, so we have
	\begin{align}\label{za}
		\|P_k(\lambda)\phi\|_{L^{2}(\mathbb{R}^n)} \leq \|\phi\|_{L^{2}(\mathbb{R}^n)}.
	\end{align}
Using the relation (\ref{aaa}) and the $L^1 - L^\infty$ estimate in the proof of Proposition 4.4.2 in \cite{Than}, we have
    \begin{align}\label{zb}
    	\|P_k(\lambda)\phi\|_{L^{\infty}(\mathbb{R}^n)} \leq |\lambda|^{\frac{n}{2}} (2k + n)^{\frac{n-1}{2}} \|\phi\|_{L^{1}(\mathbb{R}^n)}.
    \end{align}
This estimate can also be found in the proof of Proposition 1 in \cite{Manli}. Thus, the Lemma \ref{lem1} follows by interpolating (\ref{za}) and (\ref{zb}).
\end{proof}

\noindent{\bf \emph{Proof of Theorem \ref{thm4}} : } By duality argument, it is enough to show that the boundedness of the operator $\mathcal{E}_{S_{loc}}$ from $L^2(S,d\sigma_{loc})$ to $L^{\infty}_t(\mathbb{R};L^{q'}_s(\mathbb{R};L^{p'}_x(\mathbb{R}^n)))$. Equivalently, we show that the operator $\mathcal{E}_{S_{loc}}(\mathcal{E}_{S_{loc}})^*$ is bounded from $L^{1}_t(\mathbb{R};L^{q}_s(\mathbb{R};L^{p}_x(\mathbb{R}^n)))$ to $L^{\infty}_t(\mathbb{R};L^{q'}_s(\mathbb{R};L^{p'}_x(\mathbb{R}^n)))$, where $\frac{1}{p} + \frac{1}{p'} = 1$ and $\frac{1}{q} + \frac{1}{q'} = 1$.

Let $f \in \mathcal{S}(\mathbb{R}^{n+2})$. From (\ref{ak}) and (\ref{al}), we have
\begin{align*}
	\mathcal{E}_{S_{loc}} (\mathcal{E}_{S_{loc}})^* &f(x,t,s) \\
	& = \frac{1}{(2 \pi)^2}\sum_{\alpha \in \mathbb{N}^n_0}\int_{\mathbb{R^*}}  e^{-i (2|\alpha| + n)|\lambda| s} e^{-i \lambda t} \hat{f}(\alpha,\lambda, (2|\alpha|+n)|\lambda|)\hspace{2pt}\Phi_{\alpha}^\lambda (x ) \psi((2|\alpha| + n)|\lambda|) \hspace{2pt}d\lambda  \\
	& = \frac{1}{(2 \pi)^2}\sum_{\alpha \in \mathbb{N}^n_0} \frac{1}{2|\alpha| + n}  \int_{\mathbb{R^*}}  e^{-i |\lambda| s} e^{- \frac{i \lambda t}{2|\alpha| + n} } \hat{f}(\alpha,\frac{\lambda}{2|\alpha| + n}, |\lambda|)\hspace{2pt}\Phi_{\alpha}^{\frac{\lambda}{2|\alpha| + n}} (x ) \psi(|\lambda|) \hspace{2pt}d\lambda,
\end{align*}
where the last term obtained by performing the change of variables $(2|\alpha| + n)\lambda \mapsto \lambda$ in each integral. Using (\ref{def2}), (\ref{am}) and writing $a_k = \frac{1}{2k + n}$, we obtain
\begin{align}
	\mathcal{E}_{S_{loc}} (\mathcal{E}_{S_{loc}})^*f(x,t,s)  &= \frac{1}{(2 \pi)^2}\sum_{k = 0}^{\infty} \frac{1}{2 k + n} \sum_{\pm} \int_{0}^{\infty}  e^{- i \lambda s} e^{\mp i a_k  \lambda t }P_k(a_k \lambda) f^{\pm a_k\lambda, \lambda} (x) \psi(\lambda) \hspace{2pt}d\lambda \nonumber\\
	& = C \sum_{k = 0}^{\infty} \sum_{\pm} \frac{1}{2 k + n}  \mathcal{F}_{\lambda \rightarrow  s}  \left( e^{\mp i a_k  \lambda t }P_k(a_k \lambda) f^{\pm a_k\lambda, \lambda} (x) \psi_{+}(\lambda) \hspace{2pt}\right) \label{an},
\end{align}
where $\psi_{+}(\lambda) = \psi(\lambda) \bf{1}_{\lambda > 0}$. For fixed $t \in \mathbb{R}$, applying the Hausdorff-Young inequality on the right-hand side of (\ref{an}) with respect to $s-$variable, we get
\begin{align}\label{bg}
	\| \mathcal{E}_{S_{loc}} (\mathcal{E}_{S_{loc}})^*f  \|_{L^{q'}_{s}} \leq C \sum_{k = 0}^{\infty} \sum_{\pm} \frac{1}{2 k + n}  \| \psi_{+}(\lambda) e^{\mp i a_k  \lambda t }P_k(a_k \lambda) f^{\pm a_k\lambda,\lambda} (x)  \|_{L^q_{\lambda}}.
\end{align}
Now for any function $g$ defined on $\mathbb{R}^{n+1}$ and for $q' \geq p' > 2$, by Minkowski's integral inequality, gives
\begin{align}\label{au}
	\|\mathcal{F}_{\lambda \rightarrow s} g\|_{L_s^{q'} L^{p'}_{x}}  \leq 	\|\mathcal{F}_{\lambda \rightarrow s} g\|_{L^{p'}_{x}L_s^{q'} }
	 \leq C \|g\|_{L^{p'}_{x}L_\lambda^{q} }
	 \leq C \|g\|_{L_\lambda^{q} L^{p'}_{x}}.
\end{align}
In view of (\ref{au}) and (\ref{bg}), we deduce that
\begin{align*}
	\| \mathcal{E}_{S_{loc}} (\mathcal{E}_{S_{loc}})^*f  \|_{L^{\infty}_{t} L^{q'}_{s} L^{p'}_{x}} \leq C \sum_{k = 0}^{\infty} \sum_{\pm} \frac{1}{2 k + n}  \| \psi(\lambda) P_k(a_k \lambda) f^{\pm a_k\lambda,\lambda} (x)  \|_{L^q_{\lambda} L^{p'}_{x}}.
\end{align*}
But, by Lemma \ref{lem1}, we have
\begin{align*}
	\|P_k(a_k\lambda)f^{\pm a_k\lambda,\lambda}\|_{L^{p'}_x} \leq C |a_k\lambda|^{\frac{n}{2}(1-\frac{2}{p'})}(2k + n)^{\frac{n-1}{2}(1-\frac{2}{p'})} \|\mathcal{F}_{s \rightarrow -\lambda} f(\cdot,\cdot,s)\|_{L^p_x L^1_t},
\end{align*}
which implies that
\begin{align}
	\| \mathcal{E}_{S_{loc}} (\mathcal{E}_{S_{loc}})^*f  \|_{L^{\infty}_{t} L^{q'}_{s} L^{p'}_{x}} &\leq C  \sum_{k = 0}^{\infty}  \frac{1}{(2 k + n)^{1 + \frac{1}{2}(1-\frac{2}{p'})}}  \left\| \|\mathcal{F}_{s \rightarrow -\lambda} f(\cdot,\cdot,s)\|_{L^p_x L^1_t} \psi(\lambda) \lambda^{\frac{n}{2}(1-\frac{2}{p'})} \right\|_{L^q_{\lambda}} \nonumber\\
	 \vspace{.4cm}
	& \leq C \left\| \|\mathcal{F}_{s \rightarrow -\lambda} f(\cdot,\cdot,s)\|_{L^p_x L^1_t} \psi(\lambda) \lambda^{\frac{n}{2}(1-\frac{2}{p'})}   \right\|_{L^q_{\lambda}}\nonumber\\
	& \leq C \|\mathcal{F}_{s \rightarrow -\lambda} f(\cdot,\cdot, s) \|_{L^a_\lambda L^p_x L^1_t} \| \psi(\lambda) \lambda^{\frac{n}{2}(1-\frac{2}{p'})}\|_{L^b_\lambda(\mathbb{R})},\label{bh}
\end{align}
where the last step is justified by an application of H\"{o}lder's inequality in (\ref{bh}) with $a \geq 2$, $\frac{1}{a} + \frac{1}{a'} = 1$ and $\frac{1}{a} + \frac{1}{b} = \frac{1}{q}$.
Then, taking $a' = q $ and applying  the Hausdorff-Young inequality in $\lambda-$ variable, we get
\begin{align}
\| \mathcal{E}_{S_{loc}} (\mathcal{E}_{S_{loc}})^*f  \|_{L^{\infty}_{t} L^{q'}_{s} L^{p'}_{x}}  \leq  C \| \psi(\lambda) \lambda^{\frac{n}{2}(1-\frac{2}{p'})}\|_{L^b_\lambda(\mathbb{R})} \| f \|_{L^{q}_s L^p_x L^1_t}.
\end{align}\label{zy}
Thus, (\ref{at}) follows from (\ref{zy}) by Minkowski's integral inequality for all $1 \leq q \leq p < 2$.
$\hfill\square$
\begin{remark}\label{rem1}
 We consider the surfaces
 \begin{align}
 	S_{\pm} = \{(\alpha,\lambda,\nu) \in \mathbb{N}^n_0\times\mathbb{R}^*\times\mathbb{R} : \nu^2 = (2|\alpha| + n) |\lambda| , \pm\nu >0 \}.
 \end{align}
to obtain Strichartz estimate for the wave equation (\ref{ay}). The induced measure $d\sigma_{\pm}$ by the projection $\pi : \mathbb{N}_0^n \times \mathbb{R}^* \times \mathbb{R} \rightarrow\mathbb{N}_0^n \times \mathbb{R}^* $ onto the first two factors, for the surfaces $S_{\pm}$ is given by
$$\int_{S_{\pm}} \varTheta \hspace{3pt} d\sigma_{\pm} = \sum_{\alpha \in \mathbb{N}^n_0}\int_{\mathbb{R^*}} \varTheta(\alpha,\lambda,\pm \sqrt{(2|\alpha|+n)|\lambda|}) \hspace{2pt} d\lambda,$$
for any integrable function $\varTheta$ on $S_{\pm}$.

Arguing as in the proof of Theorem (\ref{thm4}), the restriction inequality (\ref{at}) can be archived for the surface $S_0 = S_{+} \cup S_{-}$ endowed with the corresponding localized measure.
\end{remark}
\section{Anisotropic Strichartz estimates }\label{sec5}
We consider the following class of functions :
	A function $f \in \mathcal{S}(\mathbb{R}^{n+1})$ is said to be frequency localized in a ball $B(0,R)$, center at $0$ of radius $R$ if there exists a smooth, even function $\psi$ supported in $(-1,1)$ and equal to $1$ near $0$ such that
	\begin{align}\label{ze}
		f = \psi(\hspace{2pt}- R^{-2} \hspace{2pt} G\hspace{2pt} )\hspace{2pt} g,
	\end{align}
for some $g \in \mathcal{S}(\mathbb{R}^{n+1})$, which equivalent to saying that for all $(\alpha,\lambda) \in \mathbb{N}^n_0\times\mathbb{R}^*$,
   \begin{align}\label{zf}
   	\hat{f}(\alpha,\lambda) = \psi(R^{-2}(|\alpha| + n)|\lambda|) \hat{g}(\alpha,\lambda).
   \end{align}
Note that (\ref{ze}) is defined using functional calculus for $G$. By construction it is clear that any function $f \in \mathcal{S}(\mathbb{R}^{n+1})$ can be approximated by frequency localized functions in $L^2$ sense. Now we are in position to prove Theorem \ref{thm5}.

\noindent{\bf \emph{Proof of Theorem \ref{thm5}}:} First, suppose $f \in \mathcal{S}(\mathbb{R}^{n+1})$ is frequency localized in the unit ball $B(0,1)$, i.e., there exists a smooth, even function $\psi$ supported in $(-1,1)$ such that $\hat{f}(\alpha,\lambda) = \psi((|\alpha| + n)|\lambda|) \hat{g}(\alpha,\lambda)$ for some $g \in \mathcal{S}(\mathbb{R}^{n+1})$. Let $\varTheta = \hat{g}\circ \pi|_S$ and the localized measure on $S$ be $d\sigma_{loc} = \psi d\sigma$ defined in (\ref{al}).  In view of (\ref{ap}) and (\ref{ak}) we can write
$$
	e^{- i s G} f (x,t) = \mathcal{E}_{S_{loc}} (\varTheta)(x,t,s).
$$
 By the restriction inequality (\ref{ao}), we have for $2 < p \leq q \leq \infty$
\begin{align}\label{ar}
     \|e^{-i s G} f\|_{L_{t}^{\infty} L_s^q L_{x}^{p}} \leq C \|\varTheta\|_{L^2(S,d\sigma_{loc})} = C \|\hat{f}\circ\pi|_S\|_{L^2(S,d\sigma)} = C \|f\|_{L^2(\mathbb{R}^{n+1})},
 \end{align}
 where the last equality is obtained by (\ref{zh}) and the Plancherel formula (\ref{zv}).

 Next, assume that $f$ is frequency localized in the ball $B(0,R)$. By (\ref{az}) one can check that the function $f_R := f \circ \delta_{R^{-1}}$ is frequency localized in $B(0,1)$ and hence applying (\ref{ar}) we get
 \begin{align}\label{bi}
 	\|e^{-i s G} f_R(x,t)\|_{L_{t}^{\infty} L_s^q L_{x}^{p}} \leq C \|f_R\|_{L^2(\mathbb{R}^{n+1})} = C R^{\frac{n}{2} + 1} \|f\|_{L^2(\mathbb{R}^{n+1})}.
 \end{align}
 Again using (\ref{ar}), we have $e^{- i s G} f_{R}(x,t) = e^{- i R^{-2}s G} f (R^{-1} x,R^{-2} t)$, thus from (\ref{bi}) we obtain
 \begin{align}\label{bj}
 	\|e^{-i s G} f\|_{L_{t}^{\infty} L_s^q L_{x}^{p}} = R^{ - \frac{2}{q} - \frac{n}{p}}\|e^{- i R^{-2}s G} f (R^{-1} x,R^{-2} t)\|_{L_{t}^{\infty} L_s^q L_{x}^{p}} \leq C R^{\frac{n + 2}{2} - \frac{2}{q} - \frac{n}{p}}\|f\|_{L^2(\mathbb{R}^{n+1})}.
 \end{align}
So, if $f$ is frequency localized in the ball $B(0,R)$, then $$\|e^{-i s G} f\|_{L_{t}^{\infty} L_s^q L_{x}^{p}}\leq C\|f\|_{L^2(\mathbb{R}^{n+1})},$$
 provided  $ \frac{2}{q} + \frac{ n}{p} = \frac{n+2}{2}$ and hence Theorem \ref{thm5} follows by density of frequency localized functions in $L^2(\mathbb{R}^{n + 1})$.
 $\hfill\square$

\noindent{\bf \emph{Proof of Theorem \ref{thm7}} : }
 Let $f, g \in \mathcal{S}(\mathbb{R}^{n+1})$ with $G^{-1/2}g \in L^2(\mathbb{R}^{n+1})$. Using (\ref{shft}) and the inversion formula (\ref{inv1}), the unique  solution of (\ref{ay}) is given by
 \begin{align}
 	u(x,t,s) = \sum_{\pm} \frac{1}{2\pi} \int_{\mathbb{R^*}} e^{-i \lambda t}\sum_{\alpha \in \mathbb{N}^{n}} e^{\mp i s \sqrt{(2|\alpha|+n)|\lambda|}} {\widehat{\varphi_{\pm}}}(\alpha,\lambda)\Phi_{\alpha}^\lambda(x) d\lambda
 \end{align}
 where $\widehat{\varphi_{\pm}} = \frac{1}{2}\left(\hat{f} \mp i \widehat{G^{-1/2}g}\right)$.

  Let the surface $S_0 = S_{+} \cup S_{-}$ endowed with the measure $d\sigma_{\pm}$, where $S_{\pm} $, $d\sigma_{\pm}$ are defined in Remark \ref{rem1} and $\varTheta = \widehat{\varphi_{\pm}}\circ \pi|_{S_{\pm}}$ on each sheet.
  Assume that $\varphi{\pm}$ are frequency localized in $B(0,1)$. Proceeding as in Theorem \ref{thm5} for the surface $(S_0,d\sigma_{\pm})$ and using (\ref{zh}), we obtain
  \begin{align}\label{ayx}
  	\|u(x,t,s)\|_{L_{t}^{\infty} L_s^q L_{x}^{p}} \leq C \|\varTheta\|_{L^2(S,d\sigma_{\pm})} = \|\widehat{\varphi_{\pm}}\|_{L^2(\mathbb{N}_0^n\times\mathbb{R^*})}  = \|\varphi_{\pm}\|_{L^2(\mathbb{R}^{n+1})}
  \end{align}for $2 < p \leq q \leq \infty $.

 If $\varphi_{\pm}$ are frequency localized in $B(0,R)$, then the functions
 $\varphi_{\pm,R} = \varphi_{\pm}\circ\delta_{R^{-1}}$
 are frequency localized in $B(0,1)$ and give rise to the solution
 $u_R(x,t,s) = u(R^{-1}x, R^{-2}t, R^{-1}s)$.
 Thus, using (\ref{ayx}) we obtain
 $$
 	\|u(x,t,s)\|_{L_{t}^{\infty} L_s^q L_{x}^{p}} \leq C R^{\frac{n + 2}{2} - \frac{1}{q} - \frac{n}{p}}\|\varphi_{\pm}\|_{L^2(\mathbb{R}^{n+1})}.
 $$
By Plancherel formula, we have
$$\|\varphi_{\pm}\|^2_{L^2(\mathbb{R}^{n+1})} = \|\varphi_{+}\|^2_{L^2(\mathbb{R}^{n+1})} + \|\varphi_{-}\|^2_{L^2(\mathbb{R}^{n+1})} = \|f\|^2_{L^2(\mathbb{R}^{n+1})} + \|G^{-1/2}g\|^2_{L^2(\mathbb{R}^{n+1})}.$$
Hence we conclude that if $f, g $ are frequency localized in $B(0,R)$, then
$$\|u(x,t,s)\|_{L_{t}^{\infty} L_s^q L_{x}^{p}}\leq C\left(\|f\|_{L^2(\mathbb{R}^{n+1})} + \|G^{-1/2}g\|_{L^2(\mathbb{R}^{n+1})} \right)$$
provided $\frac{1}{q} + \frac{n}{p} = \frac{n + 2}{2}$. Thus Theorem \ref{thm7} follows by density argument.
  $\hfill\square$
 \section{The inhomogeneous case}\label{sec6}
 Now we consider the inhomogeneous problem:
\begin{align}\label{schr1}
	i \partial_s u(x,t,s) - G u(&x,t,s) = g(x,t,s),\quad s \in \mathbb{R} ,\hspace{3pt} (x,t)\in \mathbb{R}^{n+1},
	\\ \nonumber u(x,t,0) &= f(x,t).
\end{align}
 In this case, the solution is given by the Duhamel's formula:
\begin{align}\label{av}
 	u(x,t,s) = e^{-i s G} f(x,t) -i \int_{0}^{s} e^{- i (s-s') G} g(x,t,s') ds'.
\end{align}
\begin{theorem}\label{thm6}
	Let $f \in L^2(\mathbb{R}^{n+1})$ and $g \in L^1_s(\mathbb{R};L^2_{x,t}(\mathbb{R}^{n+1})) $.  If  $(p, q)$ lies in the admissible set $A$,
	 then the solution to the problem (\ref{schr1}), $u(x,t,s) \in L^{\infty}_t(\mathbb{R};L^q_s(\mathbb{R};L^p_x(\mathbb{R}^n)))$ and satisfies the estimate
	\begin{align}\label{bl}
		\|u(x,t,s)\|_{L_{t}^{\infty} L_s^q L_{x}^{p}}\leq C\left(\|f\|_{L^2(\mathbb{R}^{n+1})} + \|g\|_{L^1_s(\mathbb{R};L^2_{x,t}(\mathbb{R}^{n+1}))}\right).
	\end{align}
\end{theorem}
\begin{proof}
Let $v(x,t,s) = i \displaystyle\int_{0}^{s} e^{- i (s-s') G} g(x,t,s') ds' $. Clearly we have
\begin{align}\label{aw}
	\|v(\cdot,\cdot,\cdot)\|_{L^\infty_tL^q_s L^p_x} \leq \int_{\mathbb{R}} \|e^{- i (\cdot) G} e^{i s' G} g(\cdot,\cdot,s')\|_{L^\infty_tL^q_s L^p_x} ds'.
\end{align}
First assume that, for all $s'$, $g(\cdot,\cdot,s')$ is frequency localized in unit ball $B(0,1)$ in $\mathbb{R}^{n+1}$. For each $s'$, using (\ref{ar}) and the unitarity of $e^{i s' G}$, (\ref{aw}) yields
\begin{align}\label{ax}
	\|v\|_{L^\infty_tL^q_s L^p_x} \leq C \int_{\mathbb{R}} \| e^{i s' G} g(\cdot,\cdot,s')\|_{L^2(\mathbb{R}^{n+1})} ds' = C \int_{\mathbb{R}} \| g(\cdot,\cdot,s')\|_{L^2(\mathbb{R}^{n+1})} ds'.
\end{align}

Now assume, for all $s$, $g(\cdot,\cdot,s)$ is frequency localized in $B(0,R)$. Letting $$g_R = R^{-2} g(\cdot,\cdot,R^{-2}s)\circ \delta_{R^{-1}} \quad \mbox{and} \quad v_R(x,s,t) = i \int_{0}^{s} e^{- i (s-s') G} g_R(x,t,s') ds' $$ we find that $g_R(\cdot,\cdot,s)$ is frequency localized in ball $B(0,1)$ for all $s$ and $v_R(x,t,s) = \\ v(R^{-1}x,R^{-2}t,R^{-2}s)$. Applying (\ref{ax}) to $g_R$ and using $\|v_R\|_{L^\infty_tL^q_s L^p_x}  =  R^{\frac{2}{q} + \frac{n}{p}} \|v\|_{L^\infty_tL^q_s L^p_x}  $ with $\|g_R\|_{L^1(\mathbb{R};L^2(\mathbb{R}^{n+1}))} = R^{\frac{n}{2} + 1} \|g\|_{L^1(\mathbb{R};L^2(\mathbb{R}^{n+1}))}$, we obtain
\begin{align}\label{bj}
	\|v\|_{L_{t}^{\infty} L_s^q L_{x}^{p}}\leq C R^{\frac{n + 2}{2} - \frac{2}{q} - \frac{n}{p}}\|g\|_{L^1_s(\mathbb{R};L^2_{x,t}(\mathbb{R}^{n+1}))}.
\end{align}
Taking $\frac{2}{q} + \frac{n}{p} = \frac{n+2}{2}$ and using density of frequency localized functions in $L^1_s(\mathbb{R};L^2_{x,t}(\mathbb{R}^{n+1}))$, (\ref{bj}) turns out to be
\begin{align}\label{bk}
	\|v\|_{L_{t}^{\infty} L_s^q L_{x}^{p}}\leq C\|g\|_{L^1_s(\mathbb{R};L^2_{x,t}(\mathbb{R}^{n+1}))},
\end{align}
holds for all $ g \in L^1(\mathbb{R};L^2(\mathbb{R}^{n+1}))$. Combining the estimate for the first term in (\ref{av}) from Theorem \ref{thm5} together with (\ref{bk}), we get (\ref{bl}).
\end{proof}

Similarly, for the inhomogeneous Grushin wave equation
\begin{align}\label{zd}
	\partial_s^2 u&(x,t,s) + Gu(x,t,s) = h(x,t,s) \quad s \in \mathbb{R} ,\hspace{3pt} (x,t)\in \mathbb{R}^{n+1}, \\
	u&(x,t,0) = f(x,t) \quad \mbox{and} \quad \partial_su(x,t,0) = g(x,t),\nonumber
\end{align}
 we obtain the following anisotropic Strichartz estimate:
\begin{theorem}\label{thm8}
		Let $f \in L^2(\mathbb{R}^{n+1}) $, $G^{-1/2}g \in L^2(\mathbb{R}^{n+1})$ and $G^{-1/2}h \in L^1_s(\mathbb{R};L^2_{x,t}(\mathbb{R}^{n+1})) $. If  $(p, q)$ lies in the admissible set
	$A_w$, then the solution $u(x,t,s)$ of the IVP (\ref{ay}) is in $ L^{\infty}_t(\mathbb{R};L^q_s(\mathbb{R};L^p_x(\mathbb{R}^n)))$ and satisfies the estimate:
	$$\|u(x,t,s)\|_{L_{t}^{\infty} L_s^q L_{x}^{p}}\leq C\left(\|f\|_{L^2(\mathbb{R}^{n+1})} + \|G^{-1/2}g\|_{L^2(\mathbb{R}^{n+1})}  + \|G^{-1/2}h\|_{L^1_s(\mathbb{R};L^2_{x,t}(\mathbb{R}^{n+1}))}\right).$$
\end{theorem}
\section{The case $p = 2$ and $1 \leq q \leq 2 $ in Theorem \ref{thm4}}\label{sec7}

\begin{proposition}\label{prop1}
	The restriction inequality  (\ref{at}) also holds for $n = 1, p = 2$ and $1 \leq q \leq 2$.
\end{proposition}
\begin{proof}
	Note that for $n = 1$,
	\begin{align}\label{bm}
			\|\mathcal{R}_{S_{loc}} f\|^2_{L^2(S,d\sigma_{loc})} = \frac{1}{(2\pi)^2}\sum_{\pm} \sum_{k=0}^{\infty}\int_{0}^{\infty} \frac{1}{2k + 1} \left\|P_k(\pm a_k \lambda) f^{\pm a_k\lambda,\lambda}\right\|^2_{L^2(\mathbb{R})} \psi(\lambda) d\lambda.
	\end{align}
		Consider the Hilbert space $L^2(\mathbb{N}_0 \times \mathbb{R^+} ; L^{2}(\mathbb{R}))$, with respect to the inner product $ \langle \tilde{\alpha},\tilde{\beta}\rangle^{'} = \displaystyle \sum_{k=0}^{\infty}\int_{\mathbb{R^+}} \langle \tilde{\alpha}(k,\lambda),\tilde{\beta}(k,\lambda)\rangle \psi(\lambda)d\lambda$, for all $\tilde{\alpha}$, $\tilde{\beta}  \in L^2(\mathbb{N}_0 \times \mathbb{R^+} ; L^{2}(\mathbb{R}))$, where $\mathbb{R}_{+}$ denote the set of all positive reals. In view of (\ref{bm}) it is enough to prove  that the operator $T$ defined on $\mathcal{S}(\mathbb{R}^3)$ by
	$$T f = \frac{1}{(2k + 1)^{\frac{1}{2}}} P_k( a_k \lambda) f^{ a_k\lambda,\lambda},$$
	is bounded from $L^{1}_t(\mathbb{R};L^{q}_s(\mathbb{R};L^{2}_x(\mathbb{R}^n)))$ into $L^2(\mathbb{N}_0 \times \mathbb{R^+} ; L^{2}(\mathbb{R}))$ or equivalently that its adjoint $T^*$ is bounded from $L^2(\mathbb{N}_0 \times \mathbb{R^+} ; L^{2}(\mathbb{R}))$ into  $L^{\infty}_t(\mathbb{R};L^{q'}_s(\mathbb{R};L^{p'}_x(\mathbb{R}^n)))$ to obtain (\ref{at}).
	
	For $\tilde{\alpha} \in L^2(\mathbb{N}_0 \times \mathbb{R^+} ; L^{2}(\mathbb{R})) $, the operator $T^*$ can be computed to be
	$$T^*(\tilde{\alpha})(x,t,s) = \sum_{k=0}^{\infty}\int_{\mathbb{R^+}} \frac{1}{(2k + 1)^{\frac{1}{2}}} e^{- i a_k \lambda t} e^{- i |\lambda| s} \hspace{2pt}   P_k(a_k \lambda) (\tilde{\alpha}(k,\lambda))(x) \psi(\lambda) d\lambda.$$
	Using Minkowski's inequality together with the Hausdorff-Young inequality (see (\ref{au})), for any fixed $t \in \mathbb{R}$, we have
	\begin{align*}
		\|T^*(\tilde{\alpha})(\cdot,t,\cdot)\|_{L_s^{q'}L^2_x} \leq C \|g\|_{L^q_\lambda L^2_x},
	\end{align*}
	where $g(x,\lambda) = \psi(\lambda) \displaystyle\sum_{k=0}^{\infty} \frac{1}{(2k + 1)^{\frac{1}{2}}} \hspace{2pt}   P_k(a_k \lambda) (\tilde{\alpha}(k,\lambda))(x) .$
	Now
	\begin{align}\label{zc}
		\|g(\cdot,\lambda)\|^2_{L^2(\mathbb{R})} &=\psi(\lambda)^2 \sum_{k,l \geq 0} \frac{1}{(2k + 1)^{\frac{1}{2}} (2l + 1)^{\frac{1}{2}}}  \langle P_k(a_k \lambda) \tilde{\alpha}(k,\lambda),P_l(a_k \lambda) \tilde{\alpha}(l,\lambda)\rangle \nonumber\\
		& \leq  \psi(\lambda)^2 \sum_{k,l \geq 0} \frac{\|\tilde{\alpha}(k,\lambda)\|_{L^2(\mathbb{R})}  \|\tilde{\alpha}(l,\lambda)\|_{L^2(\mathbb{R})}}{(2k + 1)^{\frac{1}{2}} (2l + 1)^{\frac{1}{2}}}  |\langle \Phi_k^{a_k \lambda} ,\Phi_l^{a_l \lambda} \rangle|.
	\end{align}
The assymptotic behavior of the function $\Phi_k(\frac{x}{\sqrt{2k + 1}})$ is roughly similar to the function
\begin{align}\label{zp}
 \begin{cases}
\frac{1}{\sqrt{2}(2k+1)^{\frac{1}{2}}},& \text{if } |x| < 2k + 1,\\
0,              & \text{otherwise,}.
\end{cases}
\end{align}
	(see Remark $2.6$ of \cite{Louise}), using this one can verify that $\frac{1}{(2k + 1)^{\frac{1}{2}}(2l + 1)^{\frac{1}{2}}} |\langle\Phi_k^{\frac{1}{2k +1}},\Phi_k^{\frac{1}{2l +1}}\rangle|\leq\frac{C}{\max\{k,l\} +1}$. Thus, (\ref{zc}) turns out to be
	$$\|g(\cdot,\lambda)\|^2_{L^2(\mathbb{R})} \leq C \psi(\lambda)^2 \sum_{k \leq l} \|\tilde{\alpha}(k,\lambda)\|_{L^2(\mathbb{R})}\left(\frac{1}{l+1}\sum_{k=0}^{l+1} \|\tilde{\alpha}(k,\lambda)\|_{L^2(\mathbb{R})}\right).$$
    By Hardy's inequality, we get
	$$\|g(\cdot,\lambda)\|_{L^2(\mathbb{R})} \leq C \psi(\lambda) \left(\sum_{k = 0}^{\infty}  \|\tilde{\alpha}(k,\lambda)\|^2_{L^2(\mathbb{R})}\right)^{\frac{1}{2}}.$$
	Further, applying H\"{o}lder's inequality, we have
	$$\|g\|_{L^q_\lambda L^2_x} \leq C \|\psi(\lambda)^{\frac{1}{2}}\|_{L_{\lambda}^{\frac{2-q}{2q}}(\mathbb{R^+})} \|\tilde{\alpha}\|_{L^2(\mathbb{N}_0 \times \mathbb{R^+})}.$$
\end{proof}
The assymptotic behavior of the Hermite function $\Phi_k$ plays a decisive role (see (\ref{zp})) in the proof of Proposition \ref{prop1}.
We could not find such assymptotic behavior of the higher dimensional Hermite functions $(n \geq 2)$. But we prove the restriction inequality (\ref{at}) for $n \geq 2$ and $p = 2$ for the radial functions defined below. A function $f$ on $\mathbb{R}^{n+2}$ $(~\mbox{resp.}~\mathbb{R}^{n+1})$ is said to be radial if $f(x,t,s) = f(|x|,t,s)$ $(~\mbox{resp.}~f(x,t) = f(|x|,t))$ for all $x \in \mathbb{R}^{n+2} ~\mbox{and}~ t, s \in \mathbb{R}$. If $f$ is radial on $\mathbb{R}^{n+2}$ then ${f}^{\lambda,\nu}$ is radial on $\mathbb{R}^{n}$ for any $\lambda \in \mathbb{R}^{*}$ and $\nu \in \mathbb{R}$.
 Thus by Corollary $3.4.1$ in \cite{Than} and the relation (\ref{aaa}), for all $k \in \mathbb{N}_0$ we get
\begin{align*}
	P_{2k +1}(f^{\lambda,\nu}) = 0 \quad \mbox{and}\quad P_{2k}(\lambda)(f^{\lambda,\nu}) = R_{2k}(f^{\lambda,\nu}) L^{\frac{n}{2} - 1}_k(|\lambda||x|^2) e^{-\frac{|\lambda|}{2}|x|^2},
\end{align*}
where
\begin{align*}
	R_{2k}(f^{\lambda,\nu}) = \frac{\Gamma(k+1)}{\Gamma(k + \frac{n}{2})} |\lambda|^{\frac{n}{2}}\int_{\mathbb{R}^n} f^{\lambda,\nu}(x)L^{\frac{n}{2} - 1}_k(|\lambda||x|^2) e^{-\frac{|\lambda|}{2}|x|^2} dx,
\end{align*}
and $L^{\delta}_k$ denote the Laguerre polynomials of type $\delta (> -1)$ defined by $ L^{\delta}_k(r) =\frac{1}{k!} e^{r} r^{-\delta} \frac{d^k}{dx^k}(e^{-r} r^{k+\delta})$ for $r > 0$.
\begin{proposition}\label{prop2}
	If $f \in \mathcal{S}_{rad}(\mathbb{R}^{n+2})$, the set of all radial Schwartz class functions on $\mathbb{R}^{n+2}$, then the restriction inequality  (\ref{at}) holds for $n \geq 2$, $p = 2$ and $2 \leq q \leq \infty$.
\end{proposition}

\begin{proof}
	Let $f \in \mathcal{S}_{rad}(R^{n+2})$. To prove (\ref{at}) for $n \geq 2$ and $p = 2$ (proceeding as in (\ref{bm}) for $n=1$ case), it suffices to show
	\begin{align*}
		\sum_{k = 0}^{\infty} \int_{0}^{\infty}\left (|R(k,\frac{\lambda}{4k + n},\lambda)|^2 + |R(k,\frac{-\lambda}{4k + n},\lambda)|^2\right ) \lambda^{\frac{n}{2}}\phi(\lambda) d\lambda \leq  C \|f\|^2_{L_t^1L^q_sL^p_x},
	\end{align*}
	where
	\begin{align*}
		R(k,\lambda,\nu) = \left( \frac{\Gamma(k+1)}{\Gamma(k + \frac{n}{2})(4k + n)^{\frac{n}{2} + 1}}\right)^{\frac{1}{2}} \int_{\mathbb{R}^n} f^{\lambda,\nu}(x)L^{\frac{n}{2} - 1}_k(|\lambda||x|^2) e^{-\frac{|\lambda|}{2}|x|^2} dx.
	\end{align*}
Using the assymptotic behavior of Laguerre functions in the proof of Lemma 4.2 in \cite{Muller} and proceeding as in the proof of Proposition \ref{prop1} with appropriate modifications, we get (\ref{at}) for radial Schwartz class functions on $\mathbb{R}^{n+1}$.
\end{proof}
In view of Proposition \ref{prop1} the Theorems \ref{thm5}, \ref{thm6}, \ref{thm7}, \ref{thm8} holds for $p=2, 2\leq q \leq \infty$, $n=1$. In higher dimensions $(n \geq 2)$ the Theorems \ref{thm5}, \ref{thm6}, \ref{thm7}, \ref{thm8} for $p=2, 2\leq q \leq \infty$ can be obtained for radial functions in view of Proposition \ref{prop2} and the idea used in obtaining anisotropic Strichartz estimates.
\section*{Acknowledgments}
The first author wishes to thank the Ministry of Human Resource Development, India for the  research fellowship and Indian Institute of Technology Guwahati, India for the support provided during the period of this work.


\begin{thebibliography}{99}
	\normalsize
	\baselineskip=17pt
	
	 \bibitem{Anh} T. C. Anh, J. Lee and B. K. My, On the classification of solutions to an elliptic equation involving the Grushin operator, \emph{Complex Var. Elliptic Equ.} 63, no. 5, 671-688  (2018).

	 \bibitem{Anker} J. P. Anker and V. Pierfelice, Nonlinear Schr\"{o}dinger equation on real hyperbolic spaces. \emph{Ann. Inst. H. Poincar\'{e} C Anal. Non Lin\'{e}aire} 26, 1853-1869 (2009).

      \bibitem{Bah6} H. Bahouri, J. Y. Chemin and R. Danchin, Fourier Analysis and Applications to Nonlinear Partial Differential Equations, {Grundl. Math. Wiss.}, vol. 343, Spinger, New York (2011).

      \bibitem{Bahouri1}  H. Bahouri, C. K. Fermanian, I. Gallagher, Dispersive estimates for the Schr\"{o}dinger operator on step 2 stratified Lie groups, \emph{Anal. PDE} 9, 545--574 (2016).

      \bibitem{Bahouri} H. Bahouri, P. G\'{e}rard, and C. J. Xu, Espaces de Besov et estimations de Strichartz g\'en\'eralis\'ees surle groupe de Heisenberg, \emph{J. Anal. Math.} 82, 93-118 (2000).

      \bibitem{Bahouri-local}   H. Bahouri, D. Barilari, and I. Gallagher, Local Dispersive and Strichartz Estimates for the Schr\"{o}dinger Operator on the Heisenberg Group,      \emph{Commun.Math.Res.}  9, no. 1, 1-35 (2023).

      \bibitem{Gall} H. Bahouri, D. Barilari, I. Gallagher, Strichartz Estimates and Fourier Restriction Theorems on the Heisenberg Group. \emph{J. Fourier Anal. Appl.} 27, 21 (2021).


      \bibitem{Baouendi} M. Baouendi, Sur une classe d\'op\'erateurs elliptiques deg\'en\'er\'es, \emph{Bull. Soc. Math. France} 95 , 45-87 (1967).

      \bibitem{Bour1} J. Bourgain, A remark on Schr\"{o}dinger operators. \emph{Israel J. Math.} 77, 1-16 (1992).

      \bibitem{burgain} J. Bourgain, Fourier transform restriction phenomena for certain lattice subsets and application tononlinear evolution equations I, \emph{ Geom. and Funct. Anal.} 3, 107-156 (1993).

      \bibitem{Bour2} J. Bourgain, Refinements of Strichartz inequality and applications to 2D-NLS with critical nonlinearity, \emph{Int. Math. Res. Notices} 5, 253-283 (1998).

      \bibitem{Burq}    N. Burq, P. Gerard and N. Tzvetkov, Strichartz inequalities and the nonlinear Schr\"{o}dinger equation on compact manifolds, \emph{Amer. J. Math.} 126, 569-605, 2004.

      \bibitem{Caze}  T. Cazenave, Equations de Schr\"{o}dinger non lin\'{e}aires en dimension deux, \emph{Proc. R Soc. Edinb. Sect. A} 84, 327-346 (1979).

      \bibitem{DM}  G. M. Dall\'Ara and A. Martini, A robust approach to sharp multiplier theorems for Grushin operators, \emph{Trans. Amer. Math. Soc.} 373, no. 11, 7533-7574 (2020).


    \bibitem{jyoti}   J. Dziubanski and K. Jotsaroop, On Hardy and BMO spaces for Grushin operator. \emph{J. Fourier Anal. Appl.}
    22(4), 954-995 (2016).

      \bibitem{Franchi}       B. Franchi, C. E. Guti\'{e}rrez, R. L. Wheeden, Weighted Sobolev-Poincar\'{e} inequalities for Grushin type operators, \emph{Commun. Partial Differ. Equ.} 19, 523-604 (1994).

    \bibitem{Louise} L. Gassot and M. Latocca, Probabilistic local well-posedness for the Schr\"{o}dinger equation posed for the Grushin Laplacian, \emph{J. Funct. Anal.} 283, no. 3,  109519 (2022).

      \bibitem{Gerd}  P. G\'{e}rard and S. Grellier, The cubic Szeg\"{o} equation,  \emph{ Ann. Sci. \'{E}c. Norm. Sup\'{e}r. (4)}  43, no. 5, 761-810 (2010).

      \bibitem{GV} J. Ginibre and G. Velo, Generalized Strichartz inequalities for the wave equations, \emph{J. Funct. Anal.} 133, 50-68 (1995).

      \bibitem{Grushin71}  V. V. Grushin, On a class of elliptic operators degenerate on a submanifold, \emph{Math. USSR Sbornik} 13, 155-185 (1971).

      \bibitem{Grushin70}  V. V. Grushin, On a class of hypoelliptic operators, \emph{Math. USSR Sbornik} 12, 458-476 (1970).

      \bibitem{Del} M. D. Hierro, Dispersive and Strichartz estimates on H-type groups, \emph{Studia Math} 169, 1-20 (2005).

      \bibitem{Ivanovici}      O. Ivanovici, G. Lebeau and F. Planchon, Dispersion for the wave equation inside strictly convexdomains I: the Friedlander model case, \emph{Ann. of Math. (2)} 180, 323-380 (2014).

       \bibitem{JST} K. Jotsaroop and S. Thangavelu, $L^p$ estimates for the wave equation associated to the Grushin operator, \emph{Ann. Sc. Norm. Super. Pisa Cl. Sci.}  13(5), no. 3, 775-794 (2014).


      \bibitem{KeelT} M. Keel and T. Tao, Endpoint Strichartz estimates, \emph{Am. J. Math.} 120, 955-980 (1998).

      \bibitem{Manli} H. Liu and M. Song, A restriction theorem for Grushin operators \emph{Front. Math. China} 11, 365-375 (2016).

      \bibitem{Muller} D. M\"{u}ller, A restriction theorem for the Heisenberg group, \emph{Ann. Math.} 131, 567-587 (1990).

      \bibitem{RS} R. S. Strichartz, {Restrictions of Fourier transforms to quadratic surface and decay of solutions of wave equations}, \emph{Duke Math. J.} 44,  705-714 (1977).

      \bibitem{Than} S. Thangavelu, {Lectures on Hermite and Laguerre expansions}, Mathematical notes, Princeton Univ. Press, 42 (1993).


      \bibitem{PT} P. Tomas, {A restriction theorem for the Fourier transform}, \emph{Bull. Amer. Math. Soc.} 81, 477-478  (1975).


\end{thebibliography}
\end{document}